\documentclass[10pt]{article}
\usepackage{basepreamble}
\usepackage{higman}
\usepackage{pxfonts}
\usepackage{fullpage}
\usepackage{bussproofs}

\input xy
\xyoption{all}

\title{Well Quasi-Orders and the Functional Interpretation}
\author{Thomas Powell}
\date{Preprint, \today}

\begin{document}

\maketitle

\begin{abstract}The purpose of this article is to study the role of G\"{o}del's functional interpretation in the extraction of programs from proofs in well quasi-order theory. The main focus is on the interpretation of Nash-Williams' famous minimal bad sequence construction, and the exploration of a number of much broader problems which are related to this, particularly the question of the constructive meaning of Zorn's lemma and the notion of recursion over the non-wellfounded lexicographic ordering on infinite sequences.\end{abstract}

\section{Introduction}
\label{sec-intro}

When I was invited to contribute a chapter to this volume, I felt that I should write something that would reflect, as much as possible, the extraordinary richness of the theory of well quasi-orders. Anyone present at the Dagstuhl Seminar of January 2016 would have experienced first hand how this rather innocent looking mathematical object plays a central role in so many seemingly disparate areas, ranging from proof theory, computability theory and reverse mathematics on the one hand to to term rewriting, program verification and the world of automata and formal languages on the other. While I could never do justice to such diversity in one article, my hope was to at least explore a variety of interesting problems in my own field which arise from the study of well quasi-orders.

I decided, therefore, to write an essay on G\"{o}del's functional interpretation, and the role it plays in making constructive sense of well quasi-orders. I have chosen to organise the essay around a somewhat superficial challenge, namely the development of a program which realizes Higman's lemma for boolean alphabets:
\begin{quote}\textsc{Problem.} Write a program $\Phi$ which takes as input an infinite sequence $u$ of words over a two letter alphabet, and returns a pair of indices $i<j\in\NN$ such that $u_i$ is embedded in $u_j$.\end{quote}
Of course, as long as one has proven that two such indices must exist one could simply write a program which carries out a blind search until they are found! However, I am interested in the question of how one can formally construct a subrecursive program which constitutes a computational analogue of Nash-Williams' famous minimal bad sequence construction - an elegant combinatorial idea which appears throughout well quasi-order theory.

It is important to stress that this relatively simple problem provides merely a narrative framework: My ulterior motive is to explore a number of much more elusive problems which lurk underneath. So while on the surface we will work towards the construction of our program $\Phi$, the real aim of this essay is to try to address several deeper questions, chief among them being:
\begin{enumerate}

\item\label{item-qi} What is the computational meaning of Zorn's lemma?

\item\label{item-qii} Is it possible to sensibly define recursive functionals on chain-complete partial orders?

\item\label{item-qiii} How can one describe formally extracted programs so that they can be easily understood by a human?

\end{enumerate}
Each of these questions has significance far beyond Higman's lemma, and yet the fact that they are all naturally prompted by our elementary problem is, I believe, testament to the richness inherent to the theory of well quasi-orders.

\subsection{Proof interpretations and well quasi-orders: A brief history}
\label{sec-intro-history}

In 1958, G\"{o}del published a landmark paper \cite{Goedel(1958.0)} which introduced his functional, or `Dialectica' interpretation, which he had already conceived in the 1930s as a response to Hilbert's program and his own incompleteness theorems. Initially, the functional interpretation translated Peano arithmetic to a calculus of primitive recursive functionals in all finite types known as System $\systemT$, thereby reducing the consistency of the former theory to the latter. In modern day parlance, System $\systemT$ is nothing more than a simple functional programming language which permits the construction of higher-type primitive recursive functionals. The soundness of the functional interpretation guarantees that, whenever some statement $A$ is provable in Peano arithmetic, we can extract a total functional program in $\systemT$ which witnesses its translation $A^I$.

The functional interpretation was just one of a number of techniques designed during the mid 20th century to establish relative consistency proofs. Kreisel soon observed that these techniques could be flipped on their head and viewed from a different perspective: namely as tools for extracting computational information from non-constructive proofs \cite{Kreisel(1951.0),Kreisel(1952.0)}. While the significance of this idea was not fully appreciated at the time, in recent decades the application of proof theoretic methods to extract programs from proofs has flourished, and now proof interpretations are primarily used for this purpose. Variants of G\"{o}del's functional interpretation in particular are central to the highly successful `proof mining' program pioneered by Kohlenbach \cite{Kohlenbach(2008.0)}, which has led to new quantitative results in several areas of mathematics. At the same time, the arrival of the computer has meant that the extraction of programs from proofs can be automated, and there are now proof assistants such as \textsc{Minlog} \cite{Minlog} which are dedicated to this, and which implement sophisticated refinements of the traditional proof theoretic techniques.

So where do well quasi-orders feature in all of this? 

The vast majority of proofs in `normal' mathematics use only a very small amount of set theory. Often, proofs of existential theorems in mathematics analysis which officially require choice or comprehension, use it in such a limited way that it doesn't really contribute to the complexity of extracted programs. However, the theory of well quasi-orders contains a number of key theorems which \emph{do} use choice in a crucial way, the most notorious being those such Kruskal's theorem which historically rely on variant of Nash-Williams' minimal bad sequence construction.

As a result, these theorems have become something of a focal point for research in program extraction, as canonical existential statements which come with concise, elegant, but proof theoretically non-trivial classical proofs. The question of the computational meaning of such proofs is so deep that entire theses have been dedicated to it (such as \cite{Murthy(1990.0),Seisenberger(2003.0)}). By now, even comparatively simple results like Higman's lemma have an extensive body of research devoted to them. Thus the theory of well quasi-orders has firmly established a foothold in the world of proof theory, and it is from this perspective that I study them here. 

\subsection{The origins and purpose of this chapter}
\label{sec-intro-origins}

Given the popularity of Higman's lemma among researchers in proof theory, it's perhaps important to outline my own motivation in adding yet another paper to this menagerie.

My interest in well quasi-orders began when I was a doctoral student studying G\"{o}del's functional interpretation. Paulo Oliva suggested to me that Higman's lemma might prove a useful exercise in program extraction via the functional interpretation, as up to that point this had never been done: The majority of attempts at giving a constructive proof of the lemma had utilised some form of realizability instead. So I undertook this challenge and published my work as \cite{Powell(2012.0)}.

While this indeed turned out to be a valuable for me personally, improving my own understanding of the functional interpretation and providing me with a welcome excuse to learn about well quasi-orders, in most other respects I found my work rather unsatisfactory. In order to give a computational interpretation of the instance of dependent choice used in the proof of the theorem, I resorted to the standard technique at one's disposal - a higher-type form of bar recursion. But due to the subtlety of Nash-Williams' construction, the resulting instance of bar recursion is extremely complex, leading to an extracted term whose \emph{operational} behaviour as a program is somewhat obscure, to say the least! While after a certain amount of effort I began to see what the underlying program did, this was still very difficult to describe, and I doubt that anyone who has read \cite{Powell(2012.0)} will have gained any fresh insight into the computational meaning of Higman's lemma.

I believe that the shortcomings of this paper were partly due to my own inexperience at the time, and partly due to the fact that the basic technology for extracting programs from proofs is essentially unchanged since its introduction over half a century ago. While admittedly a range of refinements have been developed, and proof mining in particular has produced a number of extremely powerful metatheorems which guarantee the extractability of low-complexity programs from proofs in specific areas of analysis, these do not really help us when comes to non-constructive proofs in well quasi-order theory which rely in an essential way on dependent choice.

In the years that followed I ended up thinking about much more general problems which were prompted from my analysis of Higman's lemma. In particular, I studied forms of higher-order recursion closely related to Nash-Williams' construction \cite{Powell(2014.0)}, and tried to develop notation systems which allow one to describe extracted programs in a more intuitive way \cite{Powell(2016.0)}. And with the announcement of this book I felt that there was an opportunity for me to revisit my original work in light of these developments.\\

\noindent As I have already emphasised, the `official' goal of building a program $\Phi$ which witnesses Higman's lemma over boolean alphabets is nothing more than an organisational device. Indeed, this is by no means the first place in which such a program has been presented, and I reiterate that real content of this essay lies in the methods which we use to obtain it, and the series of theoretical results which lead up to the final definition of $\Phi$.

Much of the technical groundwork I will present here has been done elsewhere, and this allows me to adopt a lighter style of presentation, in which my priority will be to stress the key points and skim over the heavier details. At the same time I have tried to keep everything as self-contained as possible. So, for example, the reader not familiar with G\"{o}del's functional interpretation will be given the main definition and plenty of intuition on what it means, and should be able to follow later sections without too much confusion.

In the area of program extraction, it is not uncommon to see technical achievements presented with few examples to illustrate them, and concrete case studies which give little insight into the underlying techniques on which they are based (and I have certainly been guilty of both of these at one point or another!). But my aim here is to endeavour to strike a balance between both theory and practice, and as a result I hope that this article will form a pleasant read for both specialists in proof theory as well as those with a more general interest in well quasi-orders. 

\section{Well quasi-orders and Zorn's lemma}
\label{sec-wqo}

Let's begin at the beginning, with the definition of a well quasi-order. There are numerous equivalent formulations of this concept - one of the simplest and most widely seen is the following:
\begin{definition}\label{defn-wqo}A quasi-order $(X,\preceq)$ is a set $X$ equipped with a binary relation $\preceq$ which is reflexive and transitive. It is a well quasi-order (or WQO) if it satisfies the additional property that for any infinite sequence of elements $x_0,x_1,x_2,\ldots$ there exists some $i<j$ such that $x_i\preceq x_j$.\end{definition}
It is not difficult to see that a quasi-order is a WQO iff it contains no infinite strictly decreasing chains and no infinite sequences or pairwise incomparable elements. Therefore being a WQO is a strictly stronger property than being well-founded. For example, the quasi-order $(\NN, \; | \; )$ of natural numbers ordered by divisibility is well-founded, but not a WQO. The following is perhaps slightly less obvious:
\begin{lemma}\label{res-ramsey}A quasi-order $(X,\preceq)$ is a WQO iff any infinite sequence $x_0,x_1,x_2,\ldots$ contains an infinite increasing subsequence $x_{g(0)}\preceq x_{g(1)}\preceq x_{g(2)}\preceq\ldots$ (where $g(0)<g(1)<\ldots$).\end{lemma}
\begin{proof}For the non-trivial direction, let $(X,\leq)$ be a WQO, and take some infinite sequence $x_0,x_1,\ldots$. Define $T\subseteq\NN$ by $T:\equiv\{i\in\NN\; | \; (\forall j>i)\neg(x_i\preceq x_j)\}$. Then $T$ must be finite, otherwise we would be able to construct a sequence contradicting the assumption that $X$ is a WQO. Therefore there is some $N\in\NN$ such that for all $i\geq N$ there exists some $j>i$ with $x_i\preceq x_j$, which allows us to construct our infinite increasing subsequence.\end{proof}
Given some mathematical property, such as being well quasi-ordered, we are often interested in identifying constructions which preserve that property. WQO theory is particularly rich in such results. A simple example is the following:
\begin{proposition}\label{res-dickson}If $(X,\preceq_X)$ and $(Y,\preceq_Y)$ are WQOs, then so is their cartesian product $(X\times Y,\preceq_{X\times Y})$ under the pointwise ordering.\end{proposition}

\begin{proof}Given an infinite sequence $\pair{x_0,y_0},\pair{x_1,y_1},\ldots$, consider the first component $x_0, x_1,\ldots$. Since $X$ is a WQO, by Lemma \ref{res-ramsey} there exists an infinite increasing sequence $x_{g(0)}\preceq_X x_{g(1)}\preceq_X \ldots$. Now consider the sequence $y_{g(0)},y_{g(1)},\ldots$. Since $Y$ is a WQO, there exists some $i<j$ with $y_{g(i)}\preceq_Y y_{g(j)}$. But by transitivity we also have $x_{g(i)}\preceq x_{g(j)}$, and therefore $\pair{x_{g(i)},y_{g(i)}}\preceq_{X\times Y}\pair{x_{g(j)},y_{g(j)}}$.\end{proof}
A far more subtle result, which forms the basis of this article, is the following theorem, widely known as \emph{Higman's lemma}:
\begin{theorem}[Higman's lemma \cite{Higman(1952.0)}]\label{res-higman}If $(X,\preceq)$ is a WQO, then so is $(X^\ast,\preceq_{\ast})$, the set of finite sequences over $X$ ordered under the embeddability relation, where $\seq{x_0,\ldots,x_{m-1}}\preceq_{\ast}\seq{y_0,\ldots,y_{n-1}}$ whenever there is a strictly increasing map $f$ with $x_i\preceq_{} y_{f(i)}$ for all $i<m$.\end{theorem}
A short and extremely elegant proof of Higman's lemma was given by Nash-Williams', using the so-called minimal bad sequence construction, which is a central topic of our paper.
\begin{proof}[Proof of Higman's lemma \cite{NashWilliams(1963.0)}]Suppose for contradiction that $X$ is a WQO but that there exists an infinite sequence of words $u_0,u_1,\ldots$ such that $\neg(u_i\preceq_\ast u_{j})$ for all $i<j$. We call such a sequence a `bad sequence'. Now, using the axiom of dependent choice, pick a \emph{minimal} bad sequence $v_0,v_1,\ldots$ as follows:
\begin{quote}Given that we have already constructed $\seq{v_0,\ldots,v_{k-1}}$, define $v_k$ to be such that $\seq{v_0,\ldots,v_{k-1},v_k}$ extends to some infinite bad sequence, but $\seq{v_0,\ldots,v_{k-1},a}$ does not for any $a\lhd v_k$, by which mean any prefix $a$ of $v_k$. \end{quote}
Note that such a $v_k$ exists by the minimum principle over the wellfounded prefix relation $\lhd$, together with the fact that $\seq{v_0,\ldots,v_{k-1}}$ must extend to \emph{some} bad sequence: For $k=0$ this follows from our assumption that least one bad sequence exists, while for $k>0$ it is true by construction.

Now, the crucial point is that this minimal sequence must itself be bad: If instead there were some $i<j$ with $v_i\preceq_\ast v_j$, then $\seq{v_0,\ldots,v_{j}}$ could not extend to a bad sequence, contradicting our construction. Therefore in particular each $v_n$ must be non-empty, otherwise we would trivially have $v_n=\seq{}\preceq_\ast v_{n+1}$. This means that each $v_n$ must be a concatenation of the form $\tilde v_{n}\ast \bar v_n$ where $\tilde v_n\in X^\ast$ and $\bar v_n\in X$. By Lemma \ref{res-ramsey} the sequence $\bar v_0,\bar v_1,\ldots$ contains some increasing subsequence $\bar v_{g0}\preceq \bar v_{g1}\preceq\ldots$, so let's now consider the sequence
\begin{equation*}w:=v_0,\ldots,v_{g0-1},\tilde v_{g0},\tilde v_{g0+1},\tilde v_{g0+2},\ldots\end{equation*}
Since $\tilde v_{g0}\lhd v_{g0}$, by minimality of $v$ the sequence $w$ must be good, which means that $w_i\preceq_\ast w_j$ for some $i<j$. There are three possibilities: First $j<g0$ and so $v_i=w_i\preceq_\ast w_j= v_j$, second $i<g0$ and $j=gj'$ and so $v_i\preceq_\ast\tilde v_{gj'}$ which implies that $v_i\preceq_\ast v_{gj'}$ since $\tilde v_{gj'}\lhd v_{gj'}$, and finally $i,j=gi',gj'$ and so $\tilde v_{gi'}\preceq_\ast \tilde v_{gj'}$ which implies that $v_{gi'}\preceq_\ast v_{gj'}$ since $\bar v_{gi'}\preceq \bar v_{gj'}$. In all cases we have $v_i\preceq_\ast v_j$, contradicting the fact that $v$ is bad. Hence our original assumption was false, and we can conclude that there are no bad sequence, or equivalently that $X^\ast$ is a WQO.\end{proof}

As an immediate consequence of Higman's lemma, we see that our main problem can, in theory, be solved:
\begin{corollary}\label{res-boolean}Given an infinite sequence $u$ of words over a two letter alphabet $\{0,1\}$, there exists a pair of indices $i<j$ such that $u_i$ is embedded in $u_j$.\end{corollary}

\begin{proof}The set $(\{0,1\},=)$ trivially a WQO, therefore by Higman's lemma so is $(\{0,1\}^\ast,=_\ast)$.\end{proof}

\subsection{The minimal bad sequence construction and Zorn's lemma}
\label{sec-wqo-zorn}

The existence of a minimal bad sequence in Nash-Williams' proof of Higman's lemma can be viewed in a much broader context as a particular instance of Zorn's lemma, or equivalently, as an inductive principle over chain-complete partial orders. This was first observed by Raoult \cite{Raoult(1988.0)}, and since it informs our approach to program extraction, we will explain in a little more detail what is meant by this. 

Suppose that $(Y,\sqsupseteq)$ is a chain-complete partial order, where for each non-empty chain $\gamma$ in $Y$ we fix some lower bound $\bigwedge\gamma$, which is usually taken to be the greatest lower bound if it exists (note that the fact that we talk about lower rather than upper bounds is purely cosmetic, as it sounds slightly more natural when generalising the notion of a \emph{minimal} bad sequence). The following result is essentially just the contrapositive of the principle of open induction discussed in \cite{Raoult(1988.0)}:
\begin{proposition}\label{res-open}Let $B$ be a predicate on $Y$ which satisfies the property that for any \emph{non-empty} chain $\gamma$,
\begin{equation}\label{eqn-open}(\forall x\in\gamma)B(x)\to B\left(\bigwedge\gamma\right).\end{equation}
Then whenever $B(x)$ holds for some $x\in Y$, there is some minimal $y$ such that $B(y)$ holds, but $y\sqsupset z\to \neg B(z)$.\end{proposition}

\begin{proof}Define $S:\equiv\{x\in Y\; | \; B(x)\}$. Then whenever $B(x)$ holds for some $x$, the set $S$ is chain complete: For the empty chain we just take $x$ as a lower bound, while any non-empty chain $\gamma$ in $S$ we have that $\bigwedge\gamma\in S$ by (\ref{eqn-open}). Therefore by Zorn's lemma $S$ has some minimal element $y$.  \end{proof}

Now, consider some set $X$ which comes equipped with given a strict partial order $\lhd$ on $X$ which is wellfounded. Define the lexicographic extension $\llhd$ of $\lhd$ by
\begin{equation*}u\llhd v\mbox{ \ \ iff \ \ }(\exists n)(\initSeg{u}{n}=\initSeg{v}{n}\wedge u_n\lhd v_n)\end{equation*}
where $\initSeg{u}{n}:=\seq{u_0,\ldots,u_{n-1}}$ denotes the initial segment of $u$ of length $n$. It is easy to show that $\llhd$ is also (strict) a partial order, and while $\llhd$ is not wellfounded - for example, setting $X:={0,1}$ and defining $0\lhd 1$ we would have
\begin{equation*}1,1,1,\ldots \rrhd 0,1,1,\ldots \rrhd 0,0,1,\ldots\rrhd\ldots\end{equation*}
- it is \emph{chain-complete}. In fact, given a chain $\gamma$ in $(X^\NN,\rrhd)$, we can construct its greatest lower bound by defining $v_0\in X$ to be the minimum with respect to $\rhd$ of the first components of the elements of $\gamma$, then $v_1\in X$ to be the minimum of the second components of all elements $x\in\gamma$ with $x_0=v_0$, then $v_2$ to be the minimum of the third components of all elements $x\in\gamma$ with $x_0,x_1=v_0,v_1$ and so on, and it is not difficult to show that $\bigwedge\gamma:=v$ is a greatest lower bound of $\gamma$.

Moreover, this greatest lower bound $v$ has the property that for any $n\in\NN$, there is some $x_n\in\gamma$ which agrees with $v$ on the first $n$ elements i.e. $\initSeg{v}{n}=\initSeg{x_n}{n}$. This motivates the following definition:

%
%

%
\begin{definition}\label{defn-pw}A formula $B(u)$ on infinite sequences $u\in X^\NN$ is piecewise definable, or just piecewise, if it can be expressed in the form $(\forall n)P(\initSeg{u}{n})$ for some formula $P(s)$ on finite sequences $s\in X^\ast$.\end{definition}
\begin{theorem}\label{res-ZL-math}Let $B(u)\equiv(\forall n)P(\initSeg{u}{n})$ be a piecewise formula, and suppose that $B(u)$ holds for some $u$. Then there exists some minimal `bad' sequence $v$ such that $B(v)$ holds, but $\neg B(w)$ for any $w\llhd v$.\end{theorem}

\begin{proof}Take any non-empty chain $\gamma$ such that $(\forall x\in \gamma)B(x)$, and let $v:=\bigwedge \gamma$. We want to show that $B(v)$ holds i.e. $P(\initSeg{v}{n})$ holds for all $n\in\NN$. But as observed above, for any $n$ there exists some $x_n\in\gamma$ with $\initSeg{v}{n}=\initSeg{x_n}{n}$, and $P(\initSeg{x_n}{n})$ follows from $B(x_n)$. Therefore the existence of a minimal bad sequence $v$ follows directly from Proposition \ref{res-open}. \end{proof}

Theorem \ref{res-ZL-math} is nothing more than a generalisation of the minimal-bad-sequence construction in Nash-Williams' proof of Higman's lemma: The predicate `$u$ is bad' can be expressed as $B(u):\equiv (\forall n)(\forall i<j<n)\neg (u_i\preceq_\ast u_j)$ which is clearly a piecewise formula, and so the existence of a minimal bad sequence follows as a special case of the instance of Zorn's lemma given in Proposition \ref{res-open}, where $Y:=(X^\ast)^\NN$ and $\sqsupset$ is taken to be $\rrhd$ over the lexicographic extension of the prefix order.

\subsection{Zorn's lemma as an axiom}
\label{sec-wqo-axiom}

The reason for the short digression above is to encourage the reader to think of the minimal bad sequence construction, not as a derived result which follows from dependent choice, but as an axiomatic minimum principle over the chain-complete partial order $((X^\ast)^\NN,\rrhd)$ which can be considered a weak form of Zorn's lemma, namely
\begin{equation}\label{eqn-abstract-zorn}(\exists u) B(u)\to (\exists v)(B(v)\wedge(\forall w\llhd v)\neg B(w)),\end{equation}
where $B(u)$ ranges over \emph{piecewise} formulas. Note that in this case, the premise of Zorn's lemma, namely chain-completeness of $S=\{u\in (X^\ast)^\NN\; | \; B(u)\}$, is encoded by both the premise $(\exists u) B(u)$ and the assumption that $B$ is piecewise. 

While this slight shift of emphasis from dependent choice to Zorn's lemma might not seem significant from an ordinary mathematical perspective, it completely alters the way in which we give a computational interpretation to Nash-Williams' proof of Higman's lemma, as this depends entirely on the manner in which we choose to formalise that proof. We will discuss proof interpretations and program extraction in much more detail in Section \ref{sec-goedel} below, but since the difference between dependent choice and above formulation of Zorn's lemma motivates our formal proof in the next section, it is important to at least roughly explain why this distinction matters to us here.

The extraction of programs from proofs typically works by assigning basic programs to the axioms and rules of some mathematical theory, and then constructing general programs recursively over the structure of formal proofs. Thus the programs which interpret the axioms of our theory form our basic building blocks, and the overall size and complexity of an extracted program in terms of these blocks reflects the size and complexity of the formal proof from which it was obtained.

In the early days of proof theory, when proofs interpretations were primarily used to obtain relative consistency proofs, `extracted programs' were nothing more than hypothetical objects which gave a computational interpretation to falsity, whose existence within some formal calculus was necessary but whose structure as programs was uninteresting and irrelevant. As such, it was sensible to work in a minimal axiomatic theory which was easy to reason about on a meta-level, but not necessarily convenient for extracting programs in practice. This was the approach taken by Spector \cite{Spector(1962.0)}, who extended G\"{o}del's consistency proof to full mathematical analysis by that showing countable dependent choice could be interpreted by the scheme of \emph{bar recursion} in all finite types, a form of recursion which, while elegant, can be rather abstruse when it comes to understanding its operational semantics as part of a real program.

For us, on the other hand, a proof interpretation is a \emph{tool} for extracting an \emph{actual} program from the proof of Higman's lemma whose algorithmic behaviour can be understood to some extent, as opposed to some obscure syntactical object which essentially acts as a black box. As a result, we want to work in a axiomatic system in which Nash-Williams' minimal bad sequence construction can be cleanly and concisely formalised.  So it is natural to ask whether, to this end, we can give a more \emph{direct} proof of Nash-Williams' construction which circumvents the use of bar recursion, and leads to a more intuitive extracted program. 

Our idea will be the following: Instead of taking dependent choice as a basic axiom and using this to prove the existence of a minimal-bad-sequence, we will instead take (\ref{eqn-abstract-zorn}) as a basic axiom, from which the existence of minimal-bad-sequences follows trivially. As a result, though, we will no longer be able to rely on Spector's computational interpretation of dependent choice, and will have to instead construct a new realizer for (\ref{eqn-abstract-zorn}).

\begin{figure}[h]
\begin{center}
\[\xymatrix{\mbox{dependent choice}\ar[dd]\ar@{.>}_{\mbox{\emph{(indirect)}}}[ddrr] & & \mbox{Zorn's lemma}\ar@{->}^{\mbox{\emph{(direct)}}}[dd] \\ \\ \mbox{bar recursion}\ar[rr] & & \fbox{extracted program}} \]
\end{center}
\end{figure}

It is now perhaps becoming clear to the reader why the three deeper questions outlined in the introduction emerge naturally from Higman's lemma! The construction of our direct realizer for the principle (\ref{eqn-abstract-zorn}) carried out in Sections \ref{sec-ZLa}-\ref{sec-ZLb} offers a partial solution to Question \ref{item-qi} for the particular instance of Zorn's lemma used here. Section \ref{sec-OR} will focus specifically on the special variant of recursion over $\rrhd$ that will be required in order to do all this, and will therefore in turn address Question \ref{item-qii}. Question \ref{item-qiii} is something that will be on our minds throughout the paper, and in particular influences our description of the realizing term in terms of \emph{learning procedures}.

Before we go on, it is important to observe that the idea of replacing dependent choice with some variant of Zorn's lemma has already been considered by Berger in the setting of modified realizability, in which a variant of Raoult's principle of open induction was given a direct realizability interpretation by a form of open recursion \cite{Berger(2004.0)}. Here we will give an analogous interpretation for the \emph{functional interpretation} of a principle classically equivalent open induction, and our work here differs considerably from the realizability setting in a number of crucial respects, all of which we make clear later.

\section{A formal proof of Higman's lemma}
\label{sec-formal}

As I highlighted above, in order to apply a proof interpretation to a proof, we first need to have some kind of formal representation of this proof in mind. The route from `textbook' to formal proof is no mere preprocessing step - the structure and hence usefulness of our extracted program is entirely dependent on the way in which we make precise the logical steps encoded by our textbook proof. The power of applied proof theory is due to the fact that the careful analysis of logical subtleties in formal proofs can reveal quantitative information that is not apparent from an ordinary mathematical perspective, hence the frequent characterisation of this information as being `hidden' in the proof. In our case, as emphasised already, the fact that we will formalise Nash-Williams' proof of Higman's lemma using an axiomatic form of Zorn's lemma is absolutely crucial to our approach.

When it comes to the application of proof interpretations, one encounters two rather distinct styles in the literature. In proof mining as conceived by Kohlenbach \cite{Kohlenbach(2008.0)}, the formal analysis of a proof is typically done `by hand'. Here, a proof interpretation is simply a means to an end (typically a numerical bound on e.g. a rate of convergence), and as such plays the role of a tool to be wielded by a mathematician. In contrast, for automated program extraction in proof assistants such as \textsc{Minlog} \cite{Minlog}, a proof interpretation forms a high level description of a procedure which has to be implemented, and programs are then extracted synthetically at the push of a button. In this case it goes without saying that the user must provide as input a full machine checkable formal proof, written within the confines of some predetermined logical system.

In this article, we take a somewhat mixed approach. On the one hand, this is not a paper on formalised mathematics (and I am certainly not a specialist in this area!): We are interested in a range of rather broad theoretical issues which we intend to present through focusing on the key features of Nash-Williams' proof, and in this sense our construction of a realizing term is based on the pen-and-paper style familiar in proof mining. On the other hand, the novelties of our approach are useful partially because they can \emph{in theory} be automated within a proof assistant, and so throughout we present our construction in a semi-formal manner in the hope that the reader at least believes that the main ideas could be implemented at some point.

\subsection{The logical system}
\label{sec-formal-logical}

The main logical theory in which we work will be the theory $\PAomega$ of Peano arithmetic in all finite types. Here, we define the finite types to include the base types $\BB$ and $\NN$ for booleans and natural numbers respectively, and allow the construction of product $X\times Y$, finite sequence $X^\ast$ and function types $X\to Y$. The theory $\PAomega$ is just the usual theory of Peano arithmetic, but with variables and quantifiers for objects of any type. We write $x:X$ or $x^X$ to denote that $x$ has type $X$. Equality symbols $=_\BB$ and $=_\NN$ for base types are taken as primitive, whereas equality for other types is defined inductively terms of these, so for example $s=_{X\to Y} t:\equiv (\forall x)(sx=_Y tx)$. There are various ways of treating extensionality: for reasons which we will not go into here, the functional interpretation does not interpret the axiom of extensionality - only a weak rule form - and so if the reader prefers they can take $\PAomega$ to be the weakly-extensional variant $\WEPAomega$ defined in e.g. \cite{Kohlenbach(2008.0),Troelstra(1973.0)}, although it should be stressed that extensionality is only an issue for the interpreted theory, and when it comes to \emph{verifying} our extracted programs in later sections we freely make use of full extensionality. In any case, the exact details of the logical system are not important in this paper, as our extraction of a program is not fully formal. 

What \emph{is} important is the way in which we extend our base theory in order to deal with the minimal bad sequence argument. First note that when reasoning in higher types it is essential to be able to add to our base theory the very weak axiom of choice for quantifier-free formulas:
\begin{equation*}\QFAC \; \colon \; (\forall x^X)(\exists y^Y) A_0(x,y)\to (\exists f^{X\to Y})(\forall x) A_0(x,f(x)).\end{equation*}
which as we will see is completely harmless from a computational point of view. In contrast, the crucial additional axiom we choose in order to formalise Nash-Williams' argument is the following syntactic variant of Theorem \ref{res-ZL-math} already stated as (\ref{eqn-abstract-zorn}), which we interpret as an axiom schema labelled $\ZL$,
\begin{equation*}\ZL \; \colon \; (\exists u^{\NN\to X}) B(u)\to (\exists v)(B(v)\wedge (\forall w\llhd v)\neg B(w))\end{equation*}
where $B(u)$ is understood to range over all \emph{piecewise} formulas of the form $(\forall n) P(\initSeg{u}{n})$, and $\lhd$ is some primitive recursive relation on $X$, transfinite induction over which is provable in $\PAomega$. Of course, technically we should include this additional wellfounded assumption as a premise so that $\ZL$ becomes a proper axiom schema, but we will omit it for simplicity and if the reader prefers they can just imagine that $\ZL$ is defined relative to some arbitrary but fixed $(X,\rhd)$. Both here and in Chapter \ref{sec-higman}, $X$ will actually be of the form $X^\ast$ and $\rhd$ will be nothing more than the prefix relation on finite sequences which is trivially wellfounded.  

Note that the contrapositive of $\ZL$ can be identified with open induction over the lexicographic ordering as treated in \cite{Berger(2004.0)}:
\begin{equation*}\OI \; \colon \; (\forall v)((\forall w\llhd v) U(w)\to U(v))\to (\forall u)U(u)\end{equation*}
where now $U(u)$ must be an open formula of the form $(\exists n) P(\initSeg{u}{n})$ (and so in our terminology, being piecewise is the negation of being open). In the realizability setting of \cite{Berger(2004.0)}, there is a genuine difference between $\ZL$ and $\OI$: the latter is an intuitionistic principle which can be given a direct computational interpretation via open induction, whereas the former is a non-constructive principle which cannot be realized without the use of e.g. the $A$-translation. On the other hand, for the functional interpretation combined with the negative translation, both are $\ZL$ and $\OI$ are interpreted by exactly the same term (informally, this is due to the fact that the functional interpretation of implication is much more intricate than that of realizability), so they are essentially interchangeable here. We choose $\ZL$ as primitive, because in our opinion the realizing term is a little more intuitive when viewed as an approximation to a minimal element $v$. In any case we will discuss these nuances later. For now, we proceed straight to the formal proof. 

\subsection{The formal proof}
\label{sec-formal-proof}

Suppose that $X$ is some arbitrary finite type which comes equipped with some quasi-order $\preceq$, which we take to mean some primitive recursive function $t:X\times X\to\BB$ for which reflexivity and transitivity are provable in $\PAomega$. We now introduce two predicates which represent the two equivalent definitions of a WQO which we required in Section \ref{sec-wqo}:
\begin{equation*}\begin{aligned}\WQO(\preceq)&:\equiv (\forall x^{\NN\to X})(\exists i<j)(x_i\preceq x_j)\\
		\sWQO(\preceq)&:\equiv (\forall x^{\NN\to X})(\exists g^{\NN\to\NN})(\forall i<j)(g(i)<g(j)\wedge x_{g(i)}\preceq x_{g(j)}).\end{aligned}\end{equation*}
Now, given $\preceq$ we can formally define the embeddability relation $\preceq_\ast$ as a primitive recursive function in $\preceq$, as in order to check that $a\preceq_\ast b$ we simply need to carry out a bounded search over all increasing functions $\{0,\ldots,|a|-1\}\to \{0,\ldots,|b|-1\}$, where $|a|$ denotes the length of $a$. Note that reflexivity and transitivity of $\preceq_\ast$ is easily provable from that of $\preceq$ in $\PAomega$. The main result of this section is the following:
\begin{theorem}\label{res-higman-formal}For some fixed quasi-order $\preceq$ on $X$ we have $\PAomega+\QFAC+\ZL\vdash\sWQO(\preceq)\to\WQO(\preceq_\ast)$.\end{theorem}

The basic strategy of Nash-Williams' proof is first to construct a hypothetical minimal bad sequence, then to deal with a number of simple but quite fiddly cases in order to derive a contradiction. It will be greatly helpful to us in later sections if we separate these two parts here, and prove the following numerically explicit form of the latter:
\begin{lemma}\label{res-higman-fiddly}Given a sequence $v:(X^\ast)^\NN$, define two sequences $\tilde v:(X^\ast)^\NN$ and $\bar{v}:X^\NN$ from $v$ as follows:
\begin{equation*}\tilde{v}_n,\bar{v}_n:=\begin{cases}\seq{},0_X & \mbox{if $v_n=\seq{}$}\\
\seq{x_1,\ldots,x_{k-1}},x_k & \mbox{if $v_n=\seq{x_1,\ldots,x_k}$}\end{cases}\end{equation*}
where $0_X$ denotes some canonical element of type $X$. Given in addition some function $g:\NN\to\NN$, define the sequence $w:(X^\ast)^\NN$ by
\begin{equation*}w_n:=\begin{cases}v_n & \mbox{if $n<g(0)$}\\ \tilde{v}_{g(i)} & \mbox{if $n=g(0)+i$}.\end{cases}\end{equation*} 
Suppose that there exists some $k:\NN$ such that
\begin{equation}\label{eqn-higman-min}w_{g(0)}\lhd v_{g(0)}\to (\exists i<j<k)(w_i\preceq_\ast w_j)\end{equation}
and that $g$ satisfies
\begin{equation}\label{eqn-higman-mon}(\forall i<j\leq k)(g(i)<g(j)\wedge \bar v_{g(i)}\preceq \bar v_{g(j)}).\end{equation}
Then there exists a pair of indices $i<j<g(k)+2$ such that $v_i\preceq_\ast v_j$.

\end{lemma}

\begin{proof}This is a simple case distinction that we prove in excruciating detail. First of all, we note that by induction on (\ref{eqn-higman-mon}) it follows that $i\leq g(i)$ for all $i<k$, which we use below. There are two main cases: A degenerate one where $v_{g(0)}=\seq{}$, in which case $v_{g(0)}\preceq_\ast v_{g(0)+1}$ and so we can set $i,j=g(0),g(0)+1<g(0)+2\leq g(k)+2$. For the non-degenerate case where $v_{g(0)}\neq \seq{}$ then we have $w_{g(0)}=\tilde{v}_{g(0)}\lhd v_{g(0)}$ and hence $w_{i}\preceq_\ast w_j$ for some $i<j<k$ by (\ref{eqn-higman-min}). There are three further possibilities:
\begin{enumerate}[(i)]

\item $i<j<g(0)$: Then $v_i=w_i\preceq_\ast w_j=v_j$ and $j<k\leq g(k)<g(k)+2$.

\item $i<g(0)\leq j$: Then $v_i=w_i\preceq_\ast w_j=\tilde v_{g(j')}$ where $j=g(0)+j'$. Either $v_{g(j')}=\seq{}$ and so trivially $v_{g(j')}\preceq_\ast v_{g(j')+1}$, with $j'\leq j$ and hence $g(j')+1\leq g(j)+1<g(k)+2$, or $\tilde{v}_{g(j')}\lhd v_{g(j')}$ and therefore $v_i\preceq_\ast v_{g(j')}$ with $i<g(0)\leq g(j')<g(k)$.

\item $g(0)\leq i<j$: Then $\tilde{v}_{g(i')}=w_i\preceq_\ast w_j=\tilde{v}_{g(j')}$ with $i=g(0)+i', j=g(0)+j'$. If either $\tilde{v}_{g(i')}=\seq{}$ or $\tilde{v}_{g(j')}=\seq{}$ then the result follows exactly as in part (ii), and otherwise we have ${v}_{g(i')}=\tilde{v}_{g(i')}\ast\bar{v}_{g(i')}$ and $v_{g(j')}=\tilde{v}_{g(j')}\ast\bar{v}_{g(j')}$ and since $i'<j'\leq j<k$ it follows by (\ref{eqn-higman-mon}) that $g(i')<g(j')$ and $\bar{v}_{g(i')}\preceq \bar{v}_{g(j')}$ and hence $v_{g(i')}\preceq_\ast v_{g(j')}$, and since $g(j')\leq g(j)<g(k)$ we're done.
	
\end{enumerate}

In all cases we have found some $i''<j''<g(k)+2$ with $v_{i''}\preceq_\ast v_{j''}$. \end{proof}

\begin{proof}[Proof of Theorem \ref{res-higman-formal}]First of all, let $\lhd$ denote the strict prefix relation on words, so that $a\lhd b$ iff $|a|<|b|$ and $(\forall i<|a|)(a_i=b_i)$. This is clearly wellfounded, and we can assume for argument's sake that it is decidable, which is automatically the case when $X$ is a base type. Now, define the piecewise predicate $B(u)$ on infinite sequences of words by $B(u):\equiv (\forall n) P(\initSeg{u}{n})$, where
\begin{equation*}P(s):\equiv (\forall i<j<|s|)(s_i\npreceq_\ast s_j).\end{equation*}
Suppose that $(\exists u)B(u)$. Then by $\ZL$ there exists some $v:(X^\ast)^\NN$ such that
\begin{equation*}(\ast) \ \ \ B(v)\wedge(\forall w\llhd v)\neg B(w).\end{equation*}
Now let $\tilde v,\bar v$ and $w$ be defined as in Lemma \ref{res-higman-fiddly}, where $g$ is the function satisfying
\begin{equation*}(\forall i<j)(g(i)<g(j)\wedge \bar v_{g(i)}\preceq \bar v_{g(j)})\end{equation*}
which exists by $\sWQO(\preceq)$. Then (\ref{eqn-higman-min}) holds for some $k$ by minimality of $v$, since if $w_{g(0)}\lhd v_{g(0)}$ then $w\llhd v$, and (\ref{eqn-higman-mon}) holds for \emph{any} $k$, therefore by Lemma \ref{res-higman-fiddly} there exists $i<j$ such that $v_i\preceq v_j$, contradicting $B(v)$. Therefore $(\exists u)B(u)$ is false, which implies that for all $u$ there exists some $i<j$ such that $u_i\preceq_\ast u_j$, which is $\WQO(\preceq_\ast)$. Hence we have shown that $\sWQO(\preceq)\to \WQO(\preceq_\ast)$.\end{proof}

While the derivation above is not fully formal in the sense that would be expected were we to formally extract a program using a proof assistant, in contrast to the textbook proof given in Section \ref{sec-intro} it makes explicit important quantitative information which will guide us in constructing a realizing term, as we will see later.

Now to our main problem, which is to prove that $(\{0,1\},=_\ast)$ is a WQO. From now on we will equate the two letter alphabet $\{0,1\}$ with our type $\BB$. Suppose that $\preceq$ is now just $=_\BB$, which is clearly a WQO. In order to be able to apply Higman's lemma, we need to establish $\sWQO(=_\BB)$, either by formalising Lemma \ref{res-ramsey} or by a direct argument. We choose the latter.

\begin{theorem}\label{res-bool-WQO-formal}$\PAomega+\QFAC\vdash\sWQO(=_\BB)$\end{theorem}

\begin{proof}We will first show that
\begin{equation}\label{eqn-bool} (\forall x^{\NN\to X})(\exists b^\BB)(\forall n)(\exists k\geq n)(x_k=b).\end{equation}
Fix some sequence $x:\BB^\NN$. By the law of excluded middle we have
\begin{equation*}(\exists n)(\forall k\geq n)(x_k=0)\vee (\forall n)(\exists k\geq n)(x_k=1).\end{equation*}
If the left hand side of the disjunction holds we set $b:=0$. We have that there is some $N$ such that $x_k=0$ for all $k\geq N$, and so for an arbitrary number $n$, setting $k:=\max\{N,n\}$ yields $k\geq n$ and $x_k=0$. If the right hand side holds, we set $b:=1$ and we are done by definition.
	
So we have proved (\ref{eqn-bool}). To establish $\sWQO(=_\BB)$, take some $x$ and let $b$ be such that $(\forall n)(\exists k\geq n)(x_k=b)$. By $\QFAC$ there exists some $f:\NN\to\NN$ satisfying 
\begin{equation*}(\forall n)(f(n)\geq n\wedge x_{f(n)}=b).\end{equation*}
Now, define $g:\NN\to\NN$ via primitive recursion as
\begin{equation*}g(0):=f(0) \mbox{ \ \ and \ \ } g(n+1):=f(g(n)+1).\end{equation*}		
Then it is clear that $g(n)<g(n+1)$ and $x_{g(n)}=b$, and therefore
\begin{equation*}(\forall i<j)(g(i)<g(j)\wedge x_{g(i)}=b=x_{g(j)})\end{equation*}
and we're done.\end{proof}
Now, putting together Theorems \ref{res-higman-formal} and \ref{res-bool-WQO-formal}, we have:
\begin{corollary}\label{res-boolean-formal}$\PAomega+\QFAC+\ZL\vdash\WQO(=_{\BB,\ast})$.\end{corollary}
We summarise the main structure of our proof of $\WQO(=_{\BB,\ast})$ in Figure 1. There are three main parts to the proof, each of which will be treated somewhat separately in what follows, namely:
\begin{enumerate}[(1)]
	
\item A proof of $\sWQO(=_\BB)$ given as Theorem \ref{res-bool-WQO-formal}, which uses an instance of the law of excluded middle for $\Pi^0_2$ formulas. 

\item A single instance of $\ZL$ applied to the piecewise formula $(\forall i<j)(u_i\npreceq_\ast u_j)$, set out in the main the proof of Theorem \ref{res-higman-formal}.

\item The derivation of a contradiction from $\sWQO(\preceq)$ combined with the existence of a minimal bad sequence, which is Lemma \ref{res-higman-fiddly}.	
	
\end{enumerate}
Having now introduced the theory of WQOs and given a formal proof of our main result, the remainder of the paper will be dedicated to constructing a program which finds an embedded pair of words in an arbitrary input sequence. We will introduce our main tool - G\"{o}del's functional interpretation - in the next section, then each of the three main components will be analysed in turn in Sections \ref{sec-sWQO}, \ref{sec-ZLa}-\ref{sec-ZLb} and \ref{sec-higman}, respectively.
\begin{figure}[t]\label{fig-formal}\begin{center}

\AxiomC{$\mbox{\scriptsize Theorem \ref{res-bool-WQO-formal}}$}
\noLine
\UnaryInfC{$\vdots$}
\noLine
\UnaryInfC{$\sWQO(=_\BB)$}
\AxiomC{$(\exists u)B(u)$}
\RightLabel{$\ZL$}
\UnaryInfC{$(\exists v)(B(v)\wedge(\forall w\llhd v)\neg B(w))$}
\RightLabel{$\mbox{\scriptsize Lemma \ref{res-higman-fiddly}}$}
\BinaryInfC{$\bot$}
\UnaryInfC{$\WQO(=_{\BB,\ast})$}
\DisplayProof

\end{center}\caption{A map of the formal proof}\end{figure}

\section{G\"{o}del's functional interpretation}
\label{sec-goedel}

We now put well quasi-orders aside for a moment, and introduce the second main topic of this paper: G\"{o}del's functional (or `Dialectica') interpretation. This is in itself something of a challenge for an author: The functional interpretation is one of those syntactical objects - particularly common in proof theory - whose basic definition can be given in a few lines and whose characterising theorem (in this case soundness) can be set up in a couple of pages, and yet none of this is necessarily remotely helpful in giving the unacquainted reader any real insight into what it actually does! In reality, the functional interpretation is an extraordinarily subtle idea which continues to be studied from a range of perspectives, and the fact that it forms one of the central techniques of the highly successful proof mining program is testament to its power. For a comprehensive treatment of the functional interpretation and its role in program extraction, the reader is encouraged to consult the standard textbook \cite{Kohlenbach(2008.0)}, or alternatively the shorter chapter \cite{AvFef(1998.0)}.

Nevertheless, in an effort to make this essay as accessible as possible it is important that I say something about the interpretation here. So my plan is as follows: in Sections \ref{sec-goedel-basic}-\ref{sec-goedel-interpretation} below I will begin by defining the interpretation, and will state without proof the main results on program extraction. This will all be standard material. Then in Section \ref{sec-goedel-meaning} I will employ the slightly unconventional tactic of explaining on a high level how the functional interpretation treats a series of formulas of a specific logical shape, which appear several times in the remainder of this work. Finally, in Section \ref{sec-sWQO}, I will present in quite some detail the extraction of a simple program from the proof of Theorem \ref{res-bool-WQO-formal}, which will conveniently serve simultaneously as a illustration of the functional interpretation in action and the first step in our main challenge!

\subsection{The basics}
\label{sec-goedel-basic}

In one sentence, G\"{o}del's functional interpretation is a syntactic translation which takes as input a formula $A$ in some logical theory $\mathcal{L}$ and returns a new formula $A':=(\exists x)(\forall y)\dt{A}{x}{y}$ where $x$ and $y$ are sequences of potentially higher type variables, and $\dt{A}{x}{y}$ is in some sense computationally neutral, which in this article will just mean quantifier-free and hence decidable (since characteristic functions for all quantifier-free formulas can be constructed in $\PAomega$). The idea behind the translation is that $A\leftrightarrow A'$ over some reasonable higher-type theory, but latter can be witnessed by some term in a calculus $T$. We say that the functional interpretation soundly interprets $\mathcal{L}$ in $T$, if for any formula in the language of $\mathcal{L}$ we have
\begin{equation*}\mathcal{L}\vdash A \Rightarrow\mbox{there exists some closed term $t$ of $T$ such that $\mathcal{T}\vdash\dt{A}{t}{y}$},\end{equation*}
where $\mathcal{T}$ represent some verifying theory which allows us to reason about terms in our calculus $\mathcal{T}$. Crucially, the soundness proof comes equipped with a method of constructing such a realizer $t$ from the proof of $A$. The direct approach above typically works for \emph{intuitionistic} theories $\mathcal{L}$ extended with some weak semi-classical axioms (for example Markov's principle), but for theories $\mathcal{L}_c$ based on full classical logic, we need to precompose the functional interpretation with a negative translation $A\mapsto \nt{A}$. Therefore from now on, soundness of the functional interpretation for classical theories refers to the following:
\begin{equation*}\mathcal{L}_c\vdash A \Rightarrow\mbox{there exists some closed term $t$ of $T$ such that $\mathcal{T}\vdash\dt{\nt{A}}{t}{y}$}.\end{equation*}
In this article, our $\mathcal{L}_c$ will be $\PAomega+\QFAC$, later extended with $\ZL$. But before we go further, we need to introduce our functional calculus $T$.

\subsection{The programming language}
\label{sec-goedel-prog}

Our interpreting calculus will be a standard variant of G\"{o}del's system T, extended with product and finite sequence types as with our variant of $\PAomega$. System $T$ is well-known enough that we feel no need to give a proper definition here: in any case full details can be found in many places, including the aforementioned sources \cite{AvFef(1998.0),Kohlenbach(2008.0),Troelstra(1973.0)}. In a sentence: System $T$ is a simply typed lambda calculus which allows the definition of functionals via primitive recursion in all higher types. We summarise the basic constructions of the calculus below, if only to allow the reader to become familiar with our notational conventions. We take the types of $T$ to be the same as those in our logical system $\PAomega$, namely those build from $\BB$ and $\NN$ via product, sequence and arrow types. Terms of the calculus include the following:
\begin{itemize}

\item\textbf{Functions.}  We allow the construction of terms via lambda abstraction and application: if $x:X$ and $t:Y$ then $\lambda x.t:X\to Y$, while if $t:X\to Y$ and $s:X$ then $ts: Y$, and these satisfy the usual axioms, e.g. $(\lambda x.t[x])(s)=t[x\backslash s]$.

\item \textbf{Canonical objects.} For each type $X$ we define a canonical `zero object' $0_X:X$ in the standard manner, with $0_\NN=0$, $0_\BB=0$, $0_{\textbf{1}}=()$, $0_{X\times Y}=\pair{0_X,0_Y}$, $0_{X^\ast}=\seq{}$ and $0_{X\to Y}=\lambda x.0_Y$.

\item \textbf{Products.} Given $z:X\times Y$ we often write just $z_0,z_1$ for the projections $\pi_0z:X$, $\pi_1z\in Y$. This will also be the case for sequences, where for $z:(X\times Y)^\NN$, $z_0:X^\NN$ is defined by $(z_{0})_{n}:=\pi_0z_n$ and so on. For $x:X$ and $y:Y$ we have a pairing operator $\pair{x,y}:X\times Y$. 

\item \textbf{Sequences.} As before, given $s: X^\ast$ we denote by $|s|$ the length of $s$, for $x: X$ we define $s\ast x: X^\ast$ by $\seq{s_0,\ldots,s_{k-1},x}$ i.e. the concatenation of $s$ with $x$, and use this also for the concatenation of $s$ with another finite sequence $s\ast t:X^\NN$ or an infinite sequence $s\ast\alpha:X^\NN$. For $\alpha: X^\NN$ we let $\initSeg{\alpha}{n}:=\seq{\alpha_0,\ldots,\alpha_{n-1}}$. 

\item \textbf{Recursors.} For each type we have a recursor $\rec_X$ which has the defining equations
\begin{equation*}\rec_X^{a,h}(0)=_X a \ \ \mbox{and} \ \ \rec_X^{a,h}(n+1)=_X hn(\rec^{a,h}(n)).\end{equation*}
for parameters $a:X$ and $h:\NN\to X\to X$.

\end{itemize}

Note that having access to recursors of arbitrary finite type means that along with all normal primitive recursive functions we can define e.g. the Ackermann function (using $\rec_{\NN\to\NN}$). Indeed, the closed terms of type $\NN\to\NN$ definable in $T$ are the provably recursive functions of Peano arithmetic, a fact which follows from the soundness of the functional interpretation.

There are a couple of further remarks to be made. First, we have presented system $T$ as a equational calculus, but of course we could have instead used a conversion rule $\to_T$, in which case system $T$ can be viewed as a fragment of PCF consisting only of total objects. More concretely, any term of system $T$ can be straightforwardly written as a functional program, and we encourage the reader to think of system $T$ in this manner, as a high level means of describing real programs.

Finally, in the previous section we wrote $\mathcal{T}\vdash\dt{A}{t}{y}$, which implies that $T$ also comes equipped with a logic $\mathcal{T}$ for verifying programs. There are various ways of defining the underlying logic of system $T$ - traditionally it is presented as a minimal quantifier-free calculus with an induction axiom, although alternatively we can just identify $\mathcal{T}$ with the fully extensional variant of $\PAomega$, extended with additional axioms whenever we need them. Such distinctions are more relevant for foundational issues such as relative consistency proofs, where it was the goal to make $\mathcal{T}$ as weak as possible. Here we have no such concerns, and so we reason about the correctness of our extracted programs in a fairly free manner.

\subsection{The interpretation}
\label{sec-goedel-interpretation}

The functional interpretation $\dt{A}{x}{y}$ of a formula $A$ in the language of $\PAomega$ is defined inductively over the logical structure of $A$ as follows:
\begin{enumerate}[(i)]

\item\label{item-di} $\dt{A}{}{}:\equiv A$ if $A$ is prime

\item\label{item-dii} $\dt{A\wedge B}{x,u}{y,v}:\equiv\dt{A}{x}{y}\wedge\dt{B}{u}{v}$

\item\label{item-diii} $\dt{A\vee B}{b^\BB,x,u}{y,v}:\equiv \dt{A}{x}{y}\vee_b \dt{B}{u}{v}$

\item\label{item-div} $\dt{A\to B}{f,g}{x,v}:\equiv \dt{A}{x}{gxv}\to\dt{B}{fx}{v}$

\item\label{item-dv} $\dt{\exists t^X  A(t)}{z,x}{y}:\equiv \dt{A(z)}{x}{y}$

\item\label{item-dvi} $\dt{\forall t^X A(t)}{f}{z,y}:\equiv \dt{A(z)}{fz}{y}$

\end{enumerate}
where in clause (\ref{item-diii}) we define
\begin{equation*}P\vee_b Q:\equiv (b=0\to P)\wedge (b=1\to Q).\end{equation*}
At first glance, the functional interpretation looks very much like a standard BHK interpretation, with the exception of the treatment of implication (\ref{item-div}), which is in many ways the characterising feature of the interpretation. Though this may appear to be a little mysterious, it should be viewed as the `least non-constructive' Skolemization of the formula
\begin{equation*}(\exists x)(\forall y)\dt{A}{x}{y}\to (\exists u)(\forall v)\dt{B}{u}{v}\end{equation*}
which goes via 
\begin{equation*}(\forall x)(\exists u)(\forall v)(\exists y)(\dt{A}{x}{y}\to\dt{B}{u}{v})\end{equation*}
as an intermediate step. In the original 1958 paper \cite{Goedel(1958.0)}, G\"{o}del proved that the usual first order theory of Heyting arithmetic could be soundly interpreted in System $T$. It follows directly that Peano arithmetic can also be interpreted in $T$ when precomposed with the negative translation, and in fact it is not difficult at all to extend these results to the higher-order extensions of arithmetic with quantifier-free choice:
\begin{theorem}\label{res-PA-sound}Let $A(a)$ be a formula in the language of (weakly-extensional) $\PAomega$ containing only $a$ free. Then 
\begin{equation*}\PAomega+\QFAC\vdash A(a)\Rightarrow \mathcal{T}\vdash\dt{\nt{A(a)}}{t(a)}{y}\end{equation*}
where $t$ is a closed term of $T$ which can be formally extracted from the proof of $A(a)$.\end{theorem}
A modern presentation and proof of this result can be found in \cite{Kohlenbach(2008.0)}, which also discusses the various theories $\mathcal{T}$ in which soundness can be formalised. As simple as Theorem \ref{res-PA-sound} appears, understanding how the combination of negative translation and functional interpretation treats even simple logical formulas is far from straightforward, and is often a stumbling block when one first encounters the ideas of applied proof theory. We now try to provide some insight into this.

\subsection{The \emph{meaning} of the interpretation}
\label{sec-goedel-meaning}

In order that the reader not familiar with the functional interpretation and program extraction can follow the main part of the paper, it is important that we highlight how the functional interpretation combined with the negative translation treats a handful of key formulas. Note that we have not yet stated which negative translation we use. Unless one wants to formalise the translation this is not so important: Typically what we do is take some arbitrary choice of $\nt{A}$ and rearrange it into a simpler, intuitionistically equivalent formula which is easier to interpret.

\subsubsection{$\Pi_2$ formulas $A:\equiv (\forall x)(\exists y) B(x,y)$}
\label{sec-goedel-meaning-pi2}

We begin with one of the key properties of the functional interpretation, which makes it so useful for program extraction. The negative translation of a $\Pi_2$ formula is equivalent to $(\forall x)\neg\neg (\exists y) B(x,y)$. However, it is not difficult to see that the functional interpretation translates $\neg\neg (\exists y) B(x,y)$ to $(\exists y)B(x,y)$, and so in particular the functional interpretation admits Markov's principle. Therefore
\begin{equation*}(\forall x)\neg\neg (\exists y) B(x,y)\mbox{ is translated to }(\exists f)(\forall x)B(x,fx)\end{equation*}
and therefore in theory we can extract a program directly witnessing a $\Pi_2$ formula, even when that formula is proven classically. Note that the statement $\WQO(\preceq)$ is a $\Pi_2$ formula, and so even though we use a number of non-constructive principles in its proof, we can still hope to extract a program $\Phi$ witnessing it!

\subsubsection{$\Sigma_2$ formulas $A:\equiv (\exists x)(\forall y) B(x,y)$}
\label{sec-goedel-meaning-sig2}

In contrast to $\Pi_2$ formulas, $\Sigma_2$ formulas are more problematic, as there provable $\Sigma_2$ formulas which cannot in general be directly witnessed by a computable function (the Halting problem being the classic example). Here, the negative translation is equivalent to $\neg\neg (\exists x)(\forall y) B(x,y)$, and functional interpretation acts as follows
\begin{equation*}\begin{aligned}\neg\neg (\exists x)(\forall y) B(x,y)&\mapsto \neg (\exists f)(\forall x)\neg B(x,fx)\\ &\mapsto (\forall f)(\exists x)\neg\neg B(x,fx)\\
&\mapsto (\forall f)(\exists x) B(x,fx) \ \ \ (\ast) \\
&\mapsto(\exists \Phi)(\forall f) B(\Phi f,f(\Phi f)). \end{aligned}\end{equation*}
Note that we can omit double negations in front of $B(x,y)$ as this is a quantifier-free formula. Nevertheless, the interpretation of our original formula gives us something `indirect', which in this case coincides with Kreisel's well-known `no-counterexample' interpretation (although in general the functional interpretation is different). The intuitive idea is that $f$ is a function which attempts to witness falsity of $A$ i.e. $(\forall x)(\exists y)\neg B(x,fx)$. Then the functional $\Phi$ takes any proposed `counterexample function' and shows that it must fail. Over classical logic, the existence of such a functional $\Phi$ is equivalent to the existence of some $x$ satisfying $(\forall y)B(x,y)$, but unlike $x$ it can be directly constructed.

Another way of understanding the meaning of $\Phi$ is as a program which constructs an \emph{approximation} to the non-constructive object $x$. In general, we cannot compute an $x$ which satisfies $B(x,y)$ for all $y$, but given some $f$ we can find an $x$ which satisfies $B(x,fx)$. In this setting, $f$ should be seen as a function which calibrates \emph{how good} the approximation should be. We make extensive use of this intuition later, where we explain how the functional interpretation of the minimal bad sequence construction can be viewed as the statement that arbitrarily good `approximate' minimal bad sequences exist.

We also note that throughout this paper, we will often express the interpretation of $\Pi_2$ formulas in its penultimate form $(\forall f)(\exists x) B(x,fx)$ indicated by $(\ast)$ above. This is for no other reason than that it is much easier to talk about $x$ instead of $\Phi f$, and so we avoid a lot of rather messy notation! 

\subsubsection{Classical implication $A:\equiv (\exists x)(\forall y) B(x,y)\to (\exists u)(\forall v) C(u,v)$}
\label{sec-goedel-meaning-imp}

Finally, it is important to sketch what happens when we want to infer the existence of a non-constructive object $u$ from the existence of another non-constructive object $x$. In this case, the negative translation of $A$ is intuitionistically equivalent to
\begin{equation*}(\exists x)(\forall y) B(x,y)\to \neg\neg(\exists u)(\forall v) C(u,v),\end{equation*}
Now, by the previous section, the functional interpretation of the conclusion yields
\begin{equation*}(\exists x)(\forall y) B(x,y)\to (\forall f)(\exists u) C(u,fu).\end{equation*}
and so interpreting the implication as a whole following clause (\ref{item-div}) we get
\begin{equation*}(\exists F,G)(\forall x,f)(B(x,Gfx)\to C(Fxf,f(Fxf))).\end{equation*}
In terms of our discussion above, $F$ and $G$ can be read as follows: For any given $x$, $Fx$ is a functional which computes an approximation to the conclusion of the implication i.e. $(\forall f)(\exists u) C(u,fu)$, where now it uses that $(\forall y) B(x,y)$ holds. The functional $G$ computes exactly how big the approximation of the \emph{premise} has to be in order to build an approximation of the conclusion: this is given by $Gf$. Note that whenever we have a functional $\Phi$ which builds an approximation to the premise in this way i.e. $B(\Phi(Gf),Gf(\Phi(Gf)))$ we can use it to construct an approximation to the conclusion.

While all this may sound extremely intricate, it will hopefully become clearer when we see some concrete examples in the sequel.

\section{Interpreting the proof of $\sWQO(=_\BB)$}
\label{sec-sWQO}

We now give our first illustration of the functional interpretation in action. In Theorem \ref{res-bool-WQO-formal} we showed that the statement $\sWQO(=_\BB)$ could be formalised in $\PAomega+\QFAC$, and hence by Theorem \ref{res-PA-sound} we know for sure that we can construct a program in $T$ which witnesses its functional interpretation. However, actually doing so, and ending up with a program whose behaviour can be comprehended is another matter, and in what follows we outline the philosophy emphasised later in the paper of combining formal program extraction with intuition. Note that nothing in this section is new, and if the reader prefers they can simply glance at the program we obtain in Section \ref{sec-sWQO-simp} and move straight on to Section \ref{sec-ZLa}. 

The proof of Theorem \ref{res-PA-sound} has three main components. The first is an obviously non-constructive axiom, namely the law of excluded middle for $\Sigma_2$ formulas applied to $(\exists n)(\forall k\geq n)(x_k=0)$. The second is the derivation of the auxiliary statement (\ref{eqn-bool}) from this instance of the law of excluded-middle, and the final is the derivation of $\sWQO(=_\BB)$ from (\ref{eqn-bool}) by constructing the necessary primitive recursive function. We will treat each of these in turn. Before we do so, it is worth spelling out explicitly what our goal is. First note that $\sWQO(=_\BB)$ can be equivalently formulated as the $\Pi_3$ formula
\begin{equation*}(\forall x)(\exists g)(\forall n)(\forall i<j\leq n)(g(i)<g(j)\wedge x_{g(i)}=x_{g(j)})\end{equation*}
whose negative translation is equivalent to 
\begin{equation*}(\forall x)\neg\neg(\exists g)(\forall n)(\forall i<j\leq n)(g(i)<g(j)\wedge x_{g(i)}=x_{g(j)}).\end{equation*}
Therefore, referring back to Section \ref{sec-goedel-meaning-sig2} our challenge is to produce a program $\Phi$ which takes as input $x$ together with some `counterexample functional' $\omega:(\NN\to\NN)\to\NN$ and witnesses $(\exists g)$ in the formula
\begin{equation}\label{eqn-wqob}(\forall x,\omega)(\exists g)(\forall i<j\leq \omega g)(g(i)<g(j)\wedge x_{g(i)}=x_{g(j)}).\end{equation}
In other words, while we cannot hope to effectively construct a monotone subsequence $g$ in general, we can always do the next best thing and construct an approximation to it which works for all $i<j\leq\omega g$. Then when it comes to using $\sWQO(=_\BB)$ as a lemma in the proof of $\WQO(=_{\BB,\ast})$ we will need to calibrate exactly how big this approximation needs to be, in other words construct some concrete $\omega$ such that $\WQO(=_{\BB,\ast})$ follows from (\ref{eqn-wqob}).

\subsection{The law of excluded-middle for $\Sigma^0_2$ formulas}
\label{sec-sWQO-LEM}

Our first step when interpreting a classical proof is to interpret the main non-constructive axioms which are needed. When interpreting $\sWQO(\preceq)\to\WQO(\prec)$ later, our focus will be on $\ZL$. Here we must deal with the somewhat simpler law of excluded-middle for $\Sigma^0_2$ formulas:
\begin{equation*}(\exists n)(\forall k) P(n,k)\vee(\forall m)(\exists l)\neg P(m,l).\end{equation*}
The functional interpretation Skolemizes this as the $\Sigma_1$ formula
\begin{equation*}(\exists b,n,h)(\forall k,m)[P(n,k)\vee_b P(m,hm)]\end{equation*}
and following Section \ref{sec-goedel-meaning-sig2} the interpretation of the double negation of this formula is given by
\begin{equation}\label{eqn-LEM-nd}(\forall \phi,\psi)(\exists b,n,h)[P(n,\phi bnh)\vee_b P(\psi bnh,h(\psi bnh))].\end{equation}
This already looks rather complex thanks to all the function dependencies, but the way to think of $\phi$ and $\psi$ is again as counterexample functionals which represent the quantifiers $(\forall k)$ and $(\forall m)$ respectively. Now, we have two options in front of us: We can either carefully analyse the formal derivation of the negative translation of the law of excluded-middle in intuitionistic logic, or we can take this as a starting point and try to solve (\ref{eqn-LEM-nd}) directly. We choose the latter - and it is this kind of thing that characterizes our `semi-formal' approach to program extraction.

Let's look more closely at (\ref{eqn-LEM-nd}). Our boolean $b$ is just a marker which tells us which side of the conjunction holds, so essentially what we must do is find a pair $n_L$, $n_R$ and $h_L$, $h_R$ which satisfy either $P(n_L,\phi 0n_Lh_L)$ or $\neg P(\psi 1n_Rh_R,h_R(\psi 1n_Rh_R))$. In the first case we can then define $b,n,h$ to be $0,n_L,h_L$, and in the second to be $1,n_R,h_R$. In order to do this, we want to define these so that
\begin{equation*}P(n_L,\phi 0n_Lh_L)\leftrightarrow P(\psi 1n_Rh_R,h_R(\psi 1n_Rh_R))\end{equation*}
so that $P(n_L,\phi 0n_Lh_L)\vee \neg P(\psi 1n_Rh_R,h_R(\psi 1n_Rh_R))$ follows directly from the law of excluded-middle for quantifier-free formulas. Note that we can force this equivalence to hold if 
\begin{equation*}n_L=\psi 1n_Rh_R\mbox{ \ \ \ and \ \ \ }\phi 0n_Lh_L=h_R(\psi 1n_Rh_R).\end{equation*}
So can we solve these equations? Well, the first thing we notice is that $n_R$ and $h_L$ do not depend on anything and so can be freely chosen, so we just set these to be canonical elements $n_R,h_L:=0_\NN,0_{\NN\to\NN}$ (note that this makes sense intuitively, since $n$ only plays a role in the left disjunct, and $h$ only in the right). We can then define $n_L:=\psi 10h_R$. It remains to find some $h_R$ which satisfies
\begin{equation*}h_R(\psi 10h_R)=\phi 0n_L 0=\phi 0(\psi 10h_R)0,\end{equation*}
where the latter equality follows by substituting in our definition for $n_L$. But this is easily achieved if we set $h_R:=\lambda i.\phi 0i0$. So we're done, and to summarise, (\ref{eqn-LEM-nd}) is solved by setting
\begin{equation*}b,n,h:=\begin{cases}0,\psi 10h_R,0 & \mbox{if $P(\psi 10h_R,\phi 0(\psi 10h_R)0)$}\\ 1,0,h_R & \mbox{otherwise}\end{cases}\end{equation*}
for $h_R:=\lambda i.\phi 0i0$. The reader can now easily check that this is indeed a solution by substituting it back into (\ref{eqn-LEM-nd}). Note that while we use the law of excluded-middle in a very specific way in our proof, the above would work for \emph{any} instance of $\Sigma^0_2$ law of excluded middle (in fact we don't even need the quantifiers to be of lowest type).

\subsection{Interpreting $(\forall x)(\exists b)(\forall n)(\exists k\geq n)(x_k=b)$}
\label{sec-sWQO-LEM0}

We now use the realizing term given above to witness the functional interpretation of our intermediate result $(\forall x)(\exists b)(\forall n)(\exists k\geq n)(x_k=b)$. In order to distinguish this $b$ from that of the previous section, we relabel it as $c$. Taking into account the negative translation, what we mean is to interpret is
\begin{equation*}(\forall x)\neg\neg(\exists c,f)(\forall n)(fn\geq n\wedge x_{fn}=c)\end{equation*}
and hence 
\begin{equation*}(\forall x,\xi)(\exists c,f)(f(\xi cf)\geq \xi cf\wedge x_{f(\xi cf)}=c).\end{equation*}
Again, this looks somewhat intricate, but the term $\xi cf$ simply represents the quantifier $(\forall n)$. Now, in order to prove this statement we used the law of excluded-middle for the formula $P(n,k):\equiv (k\geq n\to x_k=0)$ given some fixed sequence $x$. What we need to do is work out exactly how this was used, and following our discussion in Section \ref{sec-goedel-meaning-imp} this means realizing the implication
\begin{equation}\label{eqn-lem-imp}(\exists b,n,h)(\forall k,m)(P(n,k)\vee_b P(m,hm))\to (\forall\xi)(\exists c,f)(f(\xi cf)\geq \xi cf\wedge x_{f(\xi cf)}=c)\end{equation}
and therefore
\begin{equation*}(\forall b,n,h,\xi)(\exists k,m,c,f)[P(n,k)\vee_b P(m,hm)\to f(\xi cf)\geq \xi cf\wedge x_{f(\xi cf)}=c].\end{equation*}
This is much easier than it looks! Let us fix $b,n,h,\xi$. There are two possibilities. If $b=0$ then we must find some $k,c,f$ (we can set $m=0$) such that the conclusion follows from $P(n,k)$. It's sensible to choose $c:=0$, then it remains to find $k,f$ satisfying
\begin{equation*}(k\geq n\to x_k=0)\to f(\xi 0f)\geq \xi 0f\wedge x_{f(\xi 0f)}=0.\end{equation*}
Following our formal proof, let's define $f(i):=\max\{n,i\}$ and 
\begin{equation*}k:=f(\xi 0f)=\max\{n,\xi 0f\}=\max\{n,\xi 0(\lambda i.\max\{n,i\}))\}.\end{equation*}
Then clearly $f(\xi 0f)\geq n,\xi 0f$ and $x_{f(\xi 0f)}=0$ follows from the premise. 

In the second case $b=1$, setting $c:=1$ (and this time $k=0$) we must establish the conclusion from $\neg P(m,hm)$, i.e. find $m,f$ satisfying
\begin{equation*}hm\geq m\wedge x_{hm}=1\to f(\xi 1f)\geq \xi 1f\wedge x_{f(\xi 1f)}=1.\end{equation*}
But this more straightforward: $f:=h$ and $m:=\xi 1h$ work, so we're done. In other words, defining 
\begin{equation*}\phi_\xi 0nh:=\max\{n,\xi 0(\lambda i.\max\{n,i\}))\} \mbox{ \ \ \ and \ \ \ } \psi_\xi 1nh:=\xi 1h\end{equation*}
with $\phi_\xi 1nh=\psi_\xi 0nh=0$, we can eliminate the quantifiers $(\forall m,k)$ in (\ref{eqn-lem-imp}), and we have proven that
\begin{equation*}(\forall b,n,h,\xi)(\exists c,f)[P(n,\phi_\xi bnh)\vee_b P(\psi_\xi bnh,h(\psi_\xi bnh))\to f(\xi cf)\geq \xi cf\wedge x_{f(\xi cf)}=c]\end{equation*}
for 
\begin{equation*}c,f:=\begin{cases}0,\lambda i.\max\{n,i\} & \mbox{if $b=0$}\\ 1,h & \mbox{otherwise}.\end{cases}\end{equation*}
But we know how find $b,n,h$ which solve the premise for $\phi_\xi$ and $\psi_\xi$, so substituting those solutions in the definition above we have
\begin{equation*}(\forall\xi)(\exists c,f)(f(\xi cf)\geq \xi cf\wedge x_{f(\xi cf)}=c)\end{equation*}
for
\begin{equation}\label{eqn-cf}c,f:=\begin{cases}0,\lambda i.\max\{\psi_\xi 10h_R,i\} & \mbox{if $P(\psi_\xi 10h_R,\phi_\xi 0(\psi_\xi 10h_R)0)$}\\
1,h_R & \mbox{otherwise}\end{cases}\end{equation}
where $\phi_\xi,\phi_\xi$, $P$ and $h_R$ are defined as above.

\subsection{Simplifying the realizing term}
\label{sec-sWQO-simp}

The solution given above for finding $c$ and $f$ in $\xi$ is perfectly valid, but still somewhat tricky to understand, as it is couched in terms of the abstruse functionals which arise from our formal proof. So while an automated extraction may produce something like this, for a human being it is desirable to simplify everything and see if there is an underlying pattern.

It immediately clear by inspecting the definition (\ref{eqn-cf}) above, that there are three key terms which play a role, namely $h_R$, $\psi_\xi 10h$ and $\phi_\xi 0i0$, with substitutions $h\mapsto h_R$ and $i\mapsto \psi_\xi 10h_R$. So it makes sense to unwind each of these terms. First, notice that from the definitions of $\phi_\xi,\psi_\xi$ we have
\begin{equation*}h_R(i)=\phi_\xi 0i0=\max\{i,\xi 0(\lambda j.\max\{i,j\})\}\mbox{ \ \ \ and \ \ \ \ }\psi_\xi 10h=\xi 1h\end{equation*}
and so in particular
\begin{equation*}\begin{aligned}\psi_\xi 10h_R&=\xi 1(\lambda i.\max\{i,\xi 0(\lambda j.\max\{i,j\})\})=:a\\
\phi_\xi 0(\psi_\xi 10h_R)0&=h_R(a)\end{aligned}\end{equation*}
and our solution can already be simplified to
\begin{equation*}c,f:=\begin{cases}0,\lambda i.\max\{a,i\} & \mbox{if $x_{\max\{a,\xi 0(\lambda j.\max\{a,j\}))\}}=0$}\\ 1,\lambda i.\xi 0(\lambda j.\max\{i,j\}) & \mbox{otherwise}.\end{cases}\end{equation*}
where we use the fact that $P(a,\max\{a,\xi 0(\lambda j.\max\{a,j\}))\})\leftrightarrow x_{\max\{a,\xi 0(\lambda j.\max\{a,j\}))\}}=0$.

We now see that, far from being the syntactic mess it appeared, our realizing term can be expressed in a very natural way. By looking closer, an interesting structure emerges: Given a function $q:\NN\times\NN\to\NN$, and a pair of functions $\varepsilon,\delta:(\NN\to\NN)\to\NN$, define the pair $(\varepsilon\otimes\delta)(q):=\pair{a,b[a]}$ where
\begin{equation*}\begin{aligned}b[i]&:=\delta(\lambda j.q(i,j))\\
a&:=\varepsilon(\lambda i.q(i,b[i])).\end{aligned}\end{equation*}
This is the so-called \emph{binary product of selection functions} studied by Escard\'{o} and Oliva in \cite{EscOli(2010.0)}. Intuitively it gives a optimal play in a two player sequential game, where $q$ is assigns an outcome to each pair of moves, and $\varepsilon,\delta$ dictate the strategy of the first and second players respectively. Using this new notation, our realizer becomes
\begin{equation*}c,f:=\begin{cases}0,\lambda i.\max\{a,i\} & \mbox{if $x_{\max\{a,b[a]\}}=0$}\\ 1,\lambda i.\xi 0(\lambda j.\max\{i,j\}) & \mbox{otherwise}.\end{cases}\end{equation*}
where $\pair{a,b[a]}=(\xi 1\otimes\xi 0)(\max)$. An extension of this idea for nested sequences of the law of excluded-middle were first considered in \cite{Oliva(2006.2)}, and generalisations of the product of selection functions to so-called `unbounded' games have been used to give computational interpretations to choice principles, thereby opening up a fascinating bridge between functional interpretations and game theory \cite{EscOli(2010.0),EscOli(2011.0),EscOli(2012.1)}.

\subsection{Summary}
\label{sec-sWQO-summary}

Our aim in this section was to lead the reader through an actual example of program extraction, from a much simpler classical principle than that about to be considered below. Rather than just presenting an extracted term, our hope was to illustrate how by analysing extracted programs and applying a degree of ingenuity, one can devise descriptions of these programs which can lead to many new results. Here, the observation by Escard\'{o} and Oliva that the functional interpretation of the law of excluded-middle concealed a natural game-theoretic construction which could be extended to encompass much stronger principles led to a large body of research in a somewhat unexpected direction. Later in this paper, a notion of a \emph{learning procedure} will play somewhat analogous role to the product of selection functions above, in the sense that it will describe a very natural computational pattern that underlies our realizer, and helps us understand its behaviour.

\section{The functional interpretation of $\ZL$ - Part 1}
\label{sec-ZLa}

The basic soundness proof of the functional interpretation (Theorem \ref{res-PA-sound}) guarantees that we are able to extract a program from any proof which can be formalised in $\PAomega+\QFAC$. However, our formalisation of the minimal bad sequence construction involves something stronger, namely an instance of $\ZL$. The following three sections contain the chief novelty of our approach, namely the solution of the functional interpretation of $\ZL$ via a form of \emph{open recursion}, which will allow us in Section \ref{sec-higman} to extract a program witnessing $\WQO(=_{\BB,\ast})$.

So what exactly is the functional interpretation of $\ZL$? Let's begin by writing out the axiom in full, where now replace $B(u)$ with the piecewise formula $(\forall n)P(\initSeg{u}{n})$, and for the remainder of the paper we now assume that $P(s)$ is quantifier-free, as it is in the case of Theorem \ref{res-higman-formal}. In order to avoid nested expressions such as $P(\initSeg{\initSeg{v}{m}}{n})$ we will use the notation $\bar P(u,n):\equiv P(\initSeg{u}{n})$. Then $\ZL$ becomes
\begin{equation*}\label{eqn-ZL-p}(\exists u)(\forall n)\bar P(u,n)\to (\exists v)((\forall n)\bar P(v,n)\wedge (\forall  w\llhd v)(\exists n)\neg\bar P(w,n)).\end{equation*}
Now, there is still an additional quantifier implicit in $(\forall w\llhd v)$, but note that
\begin{equation*}(\forall w\llhd v) A(w)\leftrightarrow (\forall m,w)(w_0\lhd v_m\to A(\initSeg{v}{m}\ast w))\end{equation*}
and so $\ZL$ can be written out in a fully explicit form as
\begin{equation}\label{eqn-ZL-exp}(\exists u)(\forall n)\bar P(u,n)\to (\exists v)((\forall n)\bar P(v,n)\wedge (\forall  m,w)(w_0\lhd v_m\to (\exists n)\neg\bar P(\initSeg{v}{m}\ast w,n))).\end{equation}
Of course, we want to apply the functional interpretation to the negative translation of (\ref{eqn-ZL-exp}), which is equivalent to
\begin{equation}\label{eqn-ZL-neg}(\exists u)(\forall n)\bar P(u,n)\to \neg\neg(\exists v)((\forall n)\bar P(v,n)\wedge (\forall  m,w)(w_0\lhd v_m\to (\exists n)\neg\bar P(\initSeg{v}{m}\ast w,n))).\end{equation}
Since this is a rather intricate formula, let's break its interpretation up into pieces. Focusing on the conclusion first, and applying the interpretation under the double negation only, we obtain
\begin{equation}\label{eqn-ZL-conc-neg}\neg\neg (\exists v,\gamma^{\NN\to X^\NN \to\NN})(\forall n,m,w)(\bar P(v,n)\wedge (\underbrace{w_0\lhd v_m\to \neg\bar P(\initSeg{v}{m}\ast w,\gamma mw)}_{C(v,\gamma,m,w)}))\end{equation}
where from now on we will use the abbreviation 
\begin{equation*}C(v,\gamma,m,w):\equiv w_0\lhd v_m\to \neg\bar P(\initSeg{v}{m}\ast w,\gamma mw)\end{equation*}
as indicated in (\ref{eqn-ZL-conc-neg}). Now, applying the functional interpretation to (\ref{eqn-ZL-conc-neg}) and referring back to the discussion in Section \ref{sec-goedel-meaning-sig2} we arrive at
\begin{equation}\label{eqn-ZL-conc-nd}(\forall N,M,W)(\exists v,\gamma)(\bar P(v,Nv\gamma)\wedge C(v,\gamma,Mv\gamma,Wv\gamma))\end{equation}
where $N,M:X^\NN\to (\NN\to X^\NN\to\NN)\to\NN$ and $W:X^\NN\to (\NN\to X^\NN\to\NN)\to X^\NN$. Substituting (\ref{eqn-ZL-conc-nd}) back into (\ref{eqn-ZL-neg}) and referring to Section \ref{sec-goedel-meaning-imp} our challenge is to witness the following expression:
\begin{equation}\label{eqn-ZL-nd}(\forall u,N,M,W)(\exists n,v,\gamma)(\bar P(u,n)\to \bar P(v,Nv\gamma)\wedge C(v,\gamma,Mv\gamma,Wv\gamma))).\end{equation}
So what does the expression (\ref{eqn-ZL-nd}) \emph{intuitively} mean? In Section \ref{sec-goedel-meaning} we characterised the functional interpretation as a translation which takes fundamentally non-constructive existence statements and converts them into `approximate' existence statements, which in theory can be given a direct computational interpretation. In its original form, $\ZL$ simply states that
\begin{quote}if there exists a bad sequence $u$ then there exists a bad sequence $v$ which is minimal with respect to $\llhd$,\end{quote}
where we call $u$ bad whenever $(\forall n)\bar P(u,n)$ holds. Now, very roughly, we can read the interpreted statement (\ref{eqn-ZL-nd}) as saying something like
\begin{quote}for any sequence $u$ and counterexample functionals $N,M,W$, there exists $n,v$ and $\gamma$ such that $\bar P(u,n)$ implies that $v$ is approximately bad with respect to $N$, and $\gamma$ witnesses that it is approximately minimal with respect to $M$ and $W$.\end{quote}
When using $\ZL$ as a lemma in the proof of a $\Pi_2$ statement, as we do in Corollary \ref{res-boolean-formal}, the task of extracting a program from this proof involves calibrating exactly what kind of approximations we need.

\subsection{A rough idea}
\label{sec-ZLa-rough}

So how do we go about solving (\ref{eqn-ZL-nd}) - in other words computing a suitable $n,v$ and $\gamma$ in terms of $u,N,M$ and $W$? A natural idea might be to simply use trial and error, as follows. Given some initial sequence $u$, we could first just try $v:=u$. Let's also set $\gamma:=\gamma_u$, where $\gamma_u$ is some function that we will need to sensibly define later, and put $n:=Nu\gamma_u$. Now suppose that $\bar P(u,Nu\gamma_u)$ holds. There are two possibilities: Either $u$ is approximately minimal in the sense that $C(u,\gamma_u,Mu\gamma_u,Wu\gamma_u)$ holds, and then we're done, or $\neg C(u,\gamma_u,Mu\gamma_u,Wu\gamma_u)$ i.e.
\begin{equation*}(Wu\gamma_u)_0\lhd u_{Mug_u}\wedge \bar P(\initSeg{u}{Mu\gamma_u}\ast Wu\gamma_u,\gamma(Mu\gamma_u)(Wu\gamma_u)).\end{equation*}
But in this case, we have found a sequence $u_1:=\initSeg{u}{Mu\gamma_u}\ast Wu\gamma_u$ which is lexicographically less that $u$ and approximately bad, so could we just set $v:=u_1$ and repeat this process, generating a sequence $u\rrhd u_1\rrhd u_2\rrhd\ldots\ldots\rrhd u_k$ until we reach some $u_k$ which works? Of course, there are a lot of details to be filled in here, in particular a formal definition of $\gamma$, but the aim of Section \ref{sec-ZLb} will be to demonstrate that this informal idea \emph{does} actually work. 

However, the obvious problem we face is that we seem to be carrying out recursion over the non-wellfounded ordering $\rrhd$, and so first we must establish a set of conditions under which this kind of recursion is well-defined. This is the purpose of Section \ref{sec-OR} which follows. Before we get into the technical details, though, we want to pause for a moment and explore the general pattern hinted at above, and introduce the notion of a learning procedure, which we have alluded to several times earlier.

\subsection{Learning procedures}
\label{sec-ZLa-learning}

Our challenge in the next Sections is to take some initial sequence $u$ which is `approximately bad' and produce a $v$ which is also approximately bad, but in addition approximately minimal. For simplicity, let's forget for a moment that we're working with infinite sequences and the lexicographic ordering, and just consider a set $X$ which comes equipped with two decidable predicates $P_0(x)$ and $C_0(x)$. Of course, $P_0(x)$ intuitively represents that $x$ is approximately bad while $C_0(x)$ represents that it's approximately minimal, but here everything is greatly simplified and do not assume anything about these formulas beyond the following property, which states that if $x$ is not minimal then there must be some $y\prec x$ satisfying $P_0(y)$: 
\begin{equation}\label{eqn-minabs}(\forall x)(\neg C_0(x)\to (\exists y\prec x) P_0(y)).\end{equation}
Our aim would be to find, from any initial $x$ satisfying $P_0(x)$, some minimal $y$ satisfying $P_0(y)$ together with $C_0(y)$ i.e.
\begin{equation}\label{eqn-learn}(\forall x)(P_0(x)\to (\exists y)(P_0(y)\wedge C_0(y))).\end{equation}
It is not too hard to come up with an algorithm which takes us from a realizer of (\ref{eqn-minabs}) to a realizer of (\ref{eqn-learn}).
\begin{lemma}\label{res-learn}Suppose that $\xi:X\to X$ is a function which satisfies
\begin{equation}\label{eqn-minabs-comp}(\forall x)(\neg C_0(x)\to x\succ\xi(x)\wedge P_0(\xi(x))).\end{equation}
For any $x:X$, the \emph{learning procedure} $\lrn{\xi}{C_0}[x]$ starting at $x$ denotes the sequence $(x_i)_{i\in\NN}$ given by
\begin{equation*}x_0:=x \mbox{ \ \ \ and \ \ \ } x_{i+1}:=\begin{cases}x_i & \mbox{if $C_0(x_i)$}\\ \xi(x_i) & \mbox{otherwise}.\end{cases}\end{equation*}
Whenever $\succ$ is wellfounded, there exists some $k$ such that $C_0(x_k)$ holds, and we call the minimal such $x_k$ the limit of $\lrn{\xi}{C_0}[x]$, which we denote by $$\liml{\xi}{C_0}{x}.$$ Then the functional $\lambda x.\liml{\xi}{C_0}{x}$ is definable using wellfounded recursion over $\succ$, and realizes (\ref{eqn-learn}) in the sense that
\begin{equation}\label{eqn-learn-comp}(\forall x)(P_0(x)\to P_0(\liml{\xi}{C_0}{x})\wedge C_0(\liml{\xi}{C_0}{x})).\end{equation}
\end{lemma}

\begin{proof}To formally construct the limit, given $C_0$ and $\xi$ define the function $L_{\xi,C_0}:X\to X^\ast$ by
\begin{equation*}L_{\xi,C_0}(x):=\begin{cases}\seq{x} & \mbox{if $C_0(x)$}\\ \seq{x}\ast L_{\xi,C_0}(\xi(x)) & \mbox{otherwise},\end{cases}\end{equation*}
which is definable via wellfounded recursion over $\succ$ since the recursive call $L_{\xi,C_0}(\xi(x))$ is only made in the event that $\neg C_0(x)$ and so $x\succ \xi(x)$ by (\ref{eqn-minabs-comp}). A simple induction over the length of $L_{\xi,C_0}(x)$ then establishes that $\liml{\xi}{C_0}{x}$ is the last element of $L_{\xi,C_0}(x)$.

That the limit satisfies (\ref{eqn-learn-comp}) essentially follows from the definition. If $P_0(x_i)$ but $\neg C_0(x_i)$ then we have $x_{i+1}=\xi(x_i)$ with $x_i\succ x_{i+1}$ and $P_0(x_{i+1})$. So it follows that if $P_0(x)$ then $P_0(x_i)$ for all $i\in\NN$. Then by the existence of a limit $x_k$ satisfying $C_0(x_k)$ we're done, since we then have $P_0(x_k)\wedge C_0(x_k)$.\end{proof}

Algorithms of the above kind can be characterised as `learning procedures' because we start with some initial attempt $x_0$ for our minimal element, and either this works or it fails, in which case we replace $x_0$ with some `improved' guess $x_0\prec x_1$ which we have learned from the failure of $x_0$ and continue in this way until we have produced an attempt $x_k$ which works and satisfies $P_0(x_k)\wedge C_0(x_k)$.

The next two sections involve adapting this basic idea to the more complex situation of constructing a realizer for the functional interpretation of $\ZL$, where the predicates $P_0$ and $C_0$ will need to take into account the counterexample functionals which determine precisely what an approximation constitutes. Moreover, we will need to adapt Lemma \ref{res-learn} so that it applies to the non-wellfounded ordering $\rrhd$.

Learning procedures as described above form the main subject of the author's paper \cite{Powell(2016.0)}, which in particular contains a solution to the functional interpretation of the least element principle for wellfounded $\succ$ that essentially forms a simple version of the realizer we construct here. Moreover, learning procedures even for certain non-wellfounded orderings are discussed in \cite[Section 5]{Powell(2016.0)}, although none of this encompasses the variant of recursion over $\rrhd$ which we require below.

\section{Recursion over $\rrhd$ in the continuous functionals}
\label{sec-OR}

In order to give a functional interpretation of $\ZL$, it is necessary that we extend system $\systemT$ with some form of recursion over the relation $\rrhd$. Since $\rrhd$ is not wellfounded, it is clear that naively introducing a general recursor over $\rrhd$ will lead to problems. However, just as $\ZL$ is equivalent to an induction principle $\OI$ over $\rrhd$, which comes with the caveat that formulas must be \emph{open} (cf. Section \ref{sec-formal-logical}), we will show that we can define a recursor over $\rrhd$ which exists in continuous models of higher-type functionals, provided that we introduce an analogous restriction for the recursor. 

The notion of recursion over $\rrhd$ is not new: In particular this forms the main topic of Berger's analysis of open induction in the framework of modified realizability \cite{Berger(2004.0)}. However, the functional interpretation requires a non-trivial adaptation of these ideas, which is the main purpose of this section.

\subsection{The problem with recursion over $\rrhd$}
\label{sec-OR-problem}

We begin by highlighting why a naive lexicographic recursor does not behave in the same way as G\"{o}del's wellfounded recursors $\rec$, as identifying the problems provides some insight into how we can potentially circumvent them. Suppose that given some pair $(X,\rhd)$ where $X$ is a type and $\rhd$ a wellfounded decidable relation on $X$, together with output type $Y$, we add to our programming language $\systemT$ an open recursor $\orec_{(X,\rhd),Y}$ which has the defining equation
\begin{equation*}\orec_{(X,\rhd),Y}^H(u)=_Y Hu(\lambda n,v\; . \; \orec^H(\initSeg{u}{n}\ast v)\mbox{ if $v_0\lhd u_n$})\end{equation*}
where `$\mbox{if $v_0\lhd u_n$}$' is short for `$\mbox{if $v_0\lhd u_n$, else $0_Y$}$'. Does our recursor give rise to well-defined functionals?

Let's consider the very simple case $X=\BB$ where $b_0\rhd b_1$ only holds in the case $1\rhd 0$, and set the output type $Y:=\NN$. Define the closed functional $\Phi:(\BB^\NN\to (\NN\to \BB^\NN\to\NN)\to\NN)\to\NN$ by
\begin{equation*}\Phi H:=\orec_{(\BB,\rhd),\NN}^H(\lambda k.1).\end{equation*}
Then we can show that the type structure of all set-theoretic functionals is no longer a model of $\systemT+(\orec_{(\BB,\rhd),\NN})$. To see this, consider the functional $H:\BB^\NN\to (\NN\to \BB^\NN\to\NN)\to\NN$ defined by
\begin{equation*}Huf:=\begin{cases}1+fn(0,1,1,\ldots) & \mbox{for the least $n$ with $u_n=1$}\\ 0 & \mbox{if no such $n$ exists}.\end{cases}\end{equation*}
Suppose that $\Phi H=N$ for some natural number $N$. Then unwinding the defining equation of $\orec_{(\BB,\rhd),\NN}^H$ we get
\begin{equation*}\begin{aligned}N=\Phi H= 1+\orec^H(0,1,1,\ldots)=2+\orec^H(0,0,1,1,\ldots)=\ldots=N+1+\orec^H(\underbrace{0,\ldots,0}_{\mbox{\scriptsize $N+1$ times}},1,1,\ldots)\geq N+1,\end{aligned}\end{equation*}
a contradiction. Here it is not necessarily surprising that we run into problems. But suppose that we demand that $H$ be \emph{continuous}, in the sense that we can determine the value of $Huf$ based on a finite initial segment of $u$ and $f$. Unfortunately, it turns out that if we increase the output type to $Y:=\NN\to\NN$ then not even continuity (or indeed even computability) can save us: Let $G:\BB^\NN\to (\NN\to\BB^\NN\to\NN^\NN)\to\NN^\NN$ be defined by
\begin{equation*}Gufn:=1+fn(0,1,1,\ldots)(n+1),\end{equation*}
and let $N:\NN$ be given by $N:=\orec^G(\lambda k.0)(0)$. Then similarly to before, we have
\begin{equation*}N=1+\orec^G(0,1,1,\ldots)(1)=2+\orec^G(0,0,1,1,\ldots)(2)=\ldots=N+1+\orec^G(\underbrace{0,\ldots,0}_{\mbox{\scriptsize $N+1$ times}},1,1,\ldots)(N+1)\geq N+1,\end{equation*}
which is again inconsistent with the axioms of Peano arithmetic. So even though $N$ is a closed term in $\systemT+(\orec_{(\BB,\rhd),\NN^\NN})$ of base type, there is no natural interpretation of $N$ in even in continuous models. So what does it take to ensure that recursion over $\rrhd$ \emph{does} have an interpretation in continuous models? To this end we will discuss two possible restrictions, namely:
\begin{itemize}

\item Leave the defining equation of the recursor unchanged but restrict $Y$ to being a base type.

\item Allow $Y$ to be an arbitrary type but introduce an explicit `control functional' into the defining equation.

\end{itemize}

The former is the approach taken by Berger in \cite{Berger(2004.0)} and works well in the setting of modified realizability. However, for the functional interpretation we need a recursor whose output type $Y$ can be arbitrary, and so we appeal to the second strategy which we will describe in detail in Section \ref{sec-OR-explicit}. However, to put our solution in context, first we will quickly sketch Berger's solution.

\subsection{The continuous functionals and Berger's open recursor}
\label{sec-OR-Berger}

In order to extend functional interpretations to subsystems of mathematical analysis, it is traditionally necessary to extend the usual interpreting calculus of functionals with a strong form of recursion, which is typically only satisfiable the \emph{continuous} models. This was originally the case with Spector's bar recursion, and also here with our variants of open recursion.

In this section we assume a basic knowledge of the type structures of partial and total continuous functionals, as a full presentation here is beyond the scope of our paper. Continuous type structures of functionals were formally constructed from the 1960s onwards: The total continuous functionals being conceived simultaneously by Kleene \cite{Kleene(1959.0)} and Kreisel \cite{Kreisel(1959.0)} and the partial model by Scott in \cite{Scott(1970.0)}. Variants of the latter play an important role in domain theory, where in particular they are used to give a denotational semantics to abstract functional programming languages such as PCF. For an up-to-date presentation of these things and much more in this direction, the reader is encouraged to consult \cite{LongNor(2015.0)}.

Very roughly, the continuous functionals $\modcont_{X\to Y}$ of type $X\to Y$ consist of functionals $F$ from $X$ to $Y$ which satisfy the property that
\begin{quote}in order to determine a finite amount of information about $F(x)$ one only needs a finite amount of information about $x$,\end{quote}
where the notion of finiteness is made precise by introducing a suitable topology for each type. Note that continuity is a strictly weaker property than being \emph{computable}: In particular \emph{any} function $f:\NN\to\NN$ is continuous by definition, since $f(n)$ only depends on a natural number $n$, and natural numbers are here considered to be finite pieces of information. On the other hand, not all functionals $F:\NN^\NN\to\NN$ are continuous, in fact the continuous functionals $\modcont_{\NN^\NN\to\NN}$ of type $2$ are precisely those such that for any $\alpha:\NN^\NN$ there exists some $N$ such that for all $\beta$, if $\initSeg{\alpha}{N}=\initSeg{\beta}{N}$ then $F(\alpha)=F(\beta)$. Both of the aforementioned properties can be generalised in the following way:
\begin{enumerate}[(i)]

\item The continuous functionals $\modcont_{\NN\to X}$ consist of \emph{all} sequences $\NN\to\modcont_X$, and so in particular the type structure of continuous functionals is a model of countable dependent choice.

\item Any $F\in\modcont_{X^\NN\to\NN}$ satisfies the following property:
\begin{equation*}\CONT \; \colon \; (\forall\alpha)(\exists N)(\forall\beta)(\initSeg{\alpha}{N}=_{X^\ast}\initSeg{\beta}{N}\to F(\alpha)=F(\beta)).\end{equation*}

\end{enumerate}
Note that for $X=\NN$ this property is equivalent to $F$ being continuous, whereas for $X$ a higher type, it is strictly weaker (since $F$ could depend on an infinite amount of information from $\alpha(0)$ but still satisfy $\CONT$, for example).

The \emph{partial} continuous functionals $\modpar$ are similar to the total continuous functionals described above, with the crucial difference that they allow functionals which are undefined in places, and so the $\modpar_X$ are represented by a \emph{domains} which in particular come equipped with a bottom element $\bot$ denoting an undefined value. The partial continuous functionals have the key property that every continuous functional $X\to X$ has a continuous fixed point, which means in particular that any recursively defined functional has a natural interpretation in $\modpar$ (although this need not be total). The partial continuous functionals are related to the total continuous functionals in that $\modcont$ is the extensional collapse of the total elements of $\modpar$ \cite{Ershov(1977.0)}. What this means in practice is that in order to show that a recursively defined functional has an interpretation in $\modcont$, it is enough to show that its interpretation in $\modpar$ as a fixpoint is total. 

\begin{theorem}[Berger \cite{Berger(2004.0)}]\label{res-OR} Let $\rhd$ be a primitive recursive relation on $X$ such that wellfounded recursion over $\rhd$ is definable in system $\systemT$. Then any fixpoint of the defining equation of $\orec_{(X,\rhd),\NN}$ is total, and hence  $\orec_{(X,\rhd),\NN}$ exists in the total continuous functionals $\modcont$. \end{theorem}

\begin{proof}While in \cite[Proposition 5.1]{Berger(2004.0)} this is proven using a variant of open induction, we appeal to the classical minimal bad sequence construction, to emphasise already the deep connection with Nash-Williams' proof of Higman's lemma. Suppose for contradiction that there are total arguments $H$ and $u$ such that $\orec^H(u)$ is not total. Using dependent choice, which is valid for total objects of sequence type, construct the minimal bad sequence $v$ of total elements of type $X$ as follows:
\begin{quote}If $\seq{v_0,\ldots,v_{k-1}}$ has already been constructed, define $v_k$ to be a total element of $\modpar_X$ such that $\orec^H(\seq{v_0,\ldots,v_{k-1},v_k}\ast w)$ is not total for some total extension $w$, but $\orec^H(\seq{v_0,\ldots,v_{k-1},a}\ast w)$ is total for all total $w$ whenever $a\lhd v_k$.\end{quote} 
Now consider $\orec^H(v)=Hv(\lambda n,w\; . \; \orec^H(\initSeg{v}{n}\ast w)\mbox{ if $w_0\lhd v_n$})=H\alpha$ where 
\begin{equation*}\alpha_n:=\pair{v_n,\lambda w. \orec^H(\initSeg{v}{n}\ast w)\mbox{ if $w_0\lhd v_n$}},\end{equation*}
and note that we use a slight abuse of types here, informally identifying the type $X^\NN\to (\NN\to X^\NN\to \NN)\to\NN$ of $H$ with $(X\times (X^\NN\to\NN))^\NN\to\NN$. But by minimality of $v$, the sequence $\alpha_n$ is total, and hence $H\alpha$ is total and then by $\CONT$ applied to the total objects $H$ and $\alpha$ there exists some $N$ such that whenever $\initSeg{\alpha}{N}=\initSeg{\beta}{N}$ then $H\alpha=H\beta$. But now consider the sequence $\initSeg{v}{N}$. By construction there exists some $w$ such that $\orec^H(\initSeg{v}{N}\ast w)$ is not total. But $\orec^H(\initSeg{v}{N}\ast w)=H\beta$ for
\begin{equation*}\beta_n:=\pair{(\initSeg{v}{N}\ast w)_n,\lambda w'. \orec^H(\initSeg{\initSeg{v}{N}\ast w}{n}\ast w')\mbox{ if $w'_0\lhd (\initSeg{v}{N}\ast w)_n$} }\end{equation*}
and we have $\alpha_n=\beta_n$ for all $n<N$ and hence $H\beta=H\alpha$ which is total, a contradiction. Hence our original assumption was wrong and $\orec^H(u)$ must be total, and since $H$ and $u$ were arbitrary we have that $\orec$ is total. \end{proof}

We have given this proof in great detail as we want to compare it to the corresponding totality proof of our explicit open recursor given in the next section. We now conclude our overview of Berger's open recursion by stating the main result of \cite{Berger(2004.0)}, namely:
\begin{theorem}[Berger \cite{Berger(2004.0)}]There is a functional definable in $\systemT+(\orec_{(X,\rhd),\NN})$ such that $\Phi$ satisfies the modified realizability interpretation of the axiom of open induction $\OI$ for $\Sigma^0_1$-piecewise formulas, provably in $\PAomega+\CONT+\OI+(\orec_{(X,\rhd),\NN})$.\end{theorem}

Theorem \ref{res-ZL-main} below forms an analogue of this for the functional interpretation.

\subsection{The explicitly controlled open recursor}
\label{sec-OR-explicit}

Berger's variant of open recursion uses in an essential way the fact that total continuous functions of type $Z^\NN\to\NN$ only consider a finite initial segment of their input. In this way we avoid the problems encountered earlier in the chapter. However, as we will see, having open recursive functionals whose output type $Y$ is arbitrary is essential for the functional interpretation of $\ZL$, and Berger's variant is no longer total in this case as we cannot rely on continuity to `implicitly' control the recursion. Therefore we require some other way of ensuring that the recursor only depends on a finite initial segment of its input. We accomplish this by adding an additional parameter $F$ to the recursor which is responsible for `explicitly' controlling the recursion, in a sense that will be made clear below. We start with some definitions. 
\begin{definition}\label{defn-specs}Suppose that $\alpha:Z^\NN$ and $m:\NN$. Then the infinite sequence $\exts{\alpha}{m}:Z^\NN$ is defined by
\begin{equation*}\exts{\alpha}{m}:=\lambda n.\begin{cases}\alpha_n & \mbox{if $n<m$}\\ 0_Z & \mbox{otherwise}.\end{cases}\end{equation*}
Now suppose in addition that $F:Z^\NN\to\NN$. Then the infinite sequence $\specs{\alpha}{F}:Z^\NN$ is defined by
\begin{equation*}\specs{\alpha}{F}:=\lambda n.\begin{cases}0_Z & \mbox{if $(\exists m\leq n)(F(\exts{\alpha}{m})<m)$}\\ \alpha_n & \mbox{otherwise}.\end{cases}\end{equation*}
Note that both $\exts{\alpha}{m}$ and $\specs{\alpha}{F}$ are primitive recursively definable.\end{definition}
\begin{lemma}\label{res-specs}Given some $F:Z^\NN\to\NN$ and $\alpha:Z^\NN$, whenever there exists some $m:\NN$ such that $F(\exts{\alpha}{m})<m$ then
\begin{equation*}\specs{\alpha}{F}=\exts{\alpha}{m_0}\end{equation*}
where $m_0$ is the least such $m$. If no such $m$ exists then $\specs{\alpha}{F}=\alpha$.\end{lemma}

\begin{proof}This follows directly from the definition: For the first case, by minimality of $m_0$ we have $\specs{\alpha}{F}(n)=\alpha_n$ for all $n<m_0$, and $\specs{\alpha}{F}(n)=0_Z$ otherwise, which is exactly the definition of $\exts{\alpha}{m_0}$. \end{proof}

\begin{lemma}\label{res-spec}For any functional $F:Z^\NN\to\NN$ satisfying $\CONT$, then for each $\alpha:Z^\NN$ there exists some $m$ such that $F(\exts{\alpha}{m})<m$.\end{lemma}

\begin{proof}Suppose that $N$ is the point of continuity of $F$ which exists by $\CONT$, and define $m:=\max\{N,F\alpha+1\}$. Then $\initSeg{\exts{\alpha}{m}}{N}=\initSeg{\alpha}{m}$ since $N\leq m$, and therefore $F(\exts{\alpha}{m})=F\alpha<m$. \end{proof}
\begin{theorem}\label{res-spec-prop}Given $F:Z^\NN\to\NN$ and $\alpha:Z^\NN$, the following facts are provable assuming $\CONT$:
\begin{enumerate}[(i)]

\item\label{item-spec-propi} $\specs{\alpha}{F}=\exts{\alpha}{m_0}$ where $m_0$ satisfies $F(\exts{\alpha}{m_0})<m_0$ and is the least such number;

\item\label{item-spec-propii} for any $\beta$ satisfying $\initSeg{\alpha}{m_0}=\initSeg{\beta}{m_0}$ we have $\specs{\alpha}{F}=\specs{\beta}{F}$;

\item\label{item-spec-propiii} $\specs{\specs{\alpha}{F}}{F}=\specs{\alpha}{F}$.

\end{enumerate}\end{theorem}

\begin{proof}Part (\ref{item-spec-propi}) follows directly from Lemma \ref{res-specs} together with Lemma \ref{res-spec}. For part (\ref{item-spec-propii}), we observe that for all $n\leq m_0$ we have $\exts{\alpha}{n}=\exts{\beta}{n}$, from which it follows that the first $m_1$ satisfying $F(\exts{\beta}{m_1})<m_1$ is just $m_0$. Therefore by part (\ref{item-spec-propi}) again we have $\specs{\beta}{F}=\exts{\beta}{m_0}=\exts{\alpha}{m_0}=\specs{\alpha}{F}$. Part (\ref{item-spec-propiii}) now follows easily, since by part (\ref{item-spec-propi}) we have $\specs{\specs{\alpha}{F}}{F}=\specs{\exts{\alpha}{m_0}}{F}$, and since $\initSeg{\exts{\alpha}{m_0}}{m_0}=\initSeg{\alpha}{m_0}$ then $\specs{\exts{\alpha}{m_0}}{F}=\specs{\alpha}{F}$ by part (\ref{item-spec-propii}).\end{proof}

Now we are ready to define our `explicit' recursor. Given $H:(X\times (X^\NN\to Y))^\NN\to Y$ and $F:(X\times (X^\NN\to Y))^\NN\to\NN$, we define
\begin{equation*}\eorec^{H,F}_{(X,\rhd),Y}(u)=_Y H(\specs{\alpha}{F})\mbox{ \ \ \ for \ \ \ } \alpha:=_{(X\times (X^\NN\to Y))^\NN}\lambda n\; . \; \pair{u_n,\lambda v\; . \; \eorec^{H,F}(\initSeg{u}{n}\ast v)\mbox{ if $v_0\lhd u_n$}}.\end{equation*}
This is a form of lexicographic recursion just as before, but with the crucial difference that the recursor now comes equipped with some additional functional $F:Z^\NN\to\NN$ which determines how much of the sequence $\alpha$ is `relevant'. As soon as we have found some $m$ satisfying the condition $F(\exts{\alpha}{m})<m$ then we declare that we are not interested in $\alpha_n$ for $n\geq m$.
Our introduction of this `control' functional $F$ allows us to provide an analogue of $\CONT$ for $H$, even though the output type of $H$ is arbitrary.
\begin{lemma}\label{res-speccont}Suppose that $H:Z^\NN\to Y$ and that $F:Z^\NN\to\NN$ satisfies $\CONT$. Then $H$ satisfies the following property:
\begin{equation*}\CONT^\ast \ \colon \  (\forall\alpha)(\exists N)(\forall\beta)(\initSeg{\alpha}{N}=\initSeg{\beta}{N}\to H(\specs{\alpha}{F})=_Y H(\specs{\beta}{F})).\end{equation*}\end{lemma}

\begin{proof}Let $m_0$ be the least number satisfying $F(\exts{\alpha}{m_0})<m_0$, which exists by $\CONT$, and define $N:=m_0$. Then if $\initSeg{\alpha}{m_0}=\initSeg{\beta}{m_0}$ then $\specs{\alpha}{F}=\specs{\beta}{F}$ by part (\ref{item-spec-propii}) above, and therefore $H(\specs{\alpha}{F})=H(\specs{\beta}{F})$.\end{proof}
We can use this result to show, analogously to Theorem \ref{res-OR}, that $\eorec_{(X,\rhd),Y}$ exists in the total continuous functionals for any $Y$. 
\begin{theorem}\label{res-EOR}Let $\rhd$ be a primitive recursive relation on $X$ such that wellfounded recursion over $\rhd$ is definable in system $\systemT$. Then the fixpoint of the defining equation of $\eorec_{(X,\rhd),Y}$ is total, and hence $\eorec_{(X,\rhd),Y}$ exists in the total continuous functionals $\modcont$.\end{theorem}

\begin{proof}This follows analogously to the proof of Theorem \ref{res-OR}. Suppose for contradiction that there are total arguments $H,F$ and $u$ such that $\eorec^{H,F}(u)$ is not total. Using dependent choice, construct a minimal bad sequence $v$ as follows:
\begin{quote}If $\seq{v_0,\ldots,v_{k-1}}$ has already been constructed, define $v_k$ to be a total element of $\modpar_X$ such that $\eorec^{H,F}(\seq{v_0,\ldots,v_{k-1},v_k}\ast w)$ is not total for some total extension $w$, but $\eorec^{H,F}(\seq{v_0,\ldots,v_{k-1},a}\ast w)$ is total for all total $w$ whenever $a\lhd v_k$.\end{quote}
Now consider $\eorec^{H,F}(v)=H(\specs{\alpha}{F})$ for 
\begin{equation*}\alpha:=\lambda n\; . \; \pair{v_n,\lambda w\; . \; \eorec^{H,F}(\initSeg{v}{n}\ast w)\mbox{ if $w_0\lhd v_n$}}.\end{equation*}
Then by construction of $v$, $\alpha$ and hence $H(\specs{\alpha}{F})$ must be total, and by $\CONT^\ast$ applied to the total objects $\alpha$, $F$ and $H$ hence there exists some $N$ such that for any total $\beta:(X\times (X^\NN\to Y))^\NN$, if $\initSeg{\alpha}{N}= \initSeg{\beta}{N}$ then $H(\specs{\alpha}{F})=H(\specs{\beta}{F})$. Now consider the sequence $\initSeg{v}{N}$. By construction there exists some $w$ such that $\eorec^{H,F}(\initSeg{v}{N}\ast w)$ is not total. But $\eorec^{H,F}(\initSeg{v}{N}\ast w)=H(\specs{\beta}{F})$ where
\begin{equation*}\beta:=\lambda n.\pair{(\initSeg{v}{N}\ast w)_n,\lambda w'\; . \; \eorec^{H,F}(\initSeg{\initSeg{v}{N}\ast w}{n}\ast w')\mbox{ if $w'_0\lhd (\initSeg{v}{N}\ast w)_n$}}\end{equation*}
and we have $\beta(n)=\alpha(n)$ for all $n<N$ and hence by $\CONT^\ast$ we have $\eorec^{H,F}(\initSeg{v}{N}\ast w)=H(\specs{\beta}{F})=H(\specs{\alpha}{F})$ which is total, a contradiction. Therefore $\eorec^{H,F}(u)$ must be total, and since $H,F$ and $u$ were arbitrary total objects then $\eorec$ is total.\end{proof}

\section{The functional interpretation of $\ZL$ - Part 2}
\label{sec-ZLb}

We will now make formal the intuitive idea presented in Section \ref{sec-ZLa}. We begin by setting up an analogue of Lemma \ref{res-learn}, but this time the objects $x$ of our learning procedures are sequences $u:X^\NN$ (and so $P_0(u)$ and $C_0(u)$ are decidable predicates over $X^\NN$) and $u_{i+1}$ is defined as $\xi(\specs{u_i}{\phi})$ for some $\phi:X^\NN\to\NN$. As a result, we end up with a sequence of the form
\begin{equation*}u_0\mapsto\specs{u_0}{\phi}\rrhd u_1\mapsto\specs{u_1}{\phi}\rrhd u_2\mapsto\ldots\end{equation*}
and so in order to guarantee that $P_0(u_i)$ holds for all $i$ we will require an additional condition, namely that the property $P_0$ is preserved under the map $\specs{\cdot}{\phi}:X^\NN\to X^\NN$. We will now just state and prove the result, but the reader is strongly encouraged to simultaneously refer back to the much simpler Lemma \ref{res-learn} and its proof, not only so that it is easier to grasp what is going on here, but because the differences in the formulation of the two lemmas are extremely informative. 

\begin{remark}\label{rem-zero}For the remainder of the paper, we request that the canonical element $0_X$ of type $X$ is minimal with respect to $\rhd$. This condition is not essential and could be circumvented by other means, but it makes what follows a little easier and allows us to avoid some additional syntax. In practice this assumption is completely benign, and in particular in Chapter \ref{sec-higman} where our type $X$ will actually be a type $X^\ast$ of finite words and $\lhd$ will denote the prefix relation, then the normal choice of $0_{X^\ast}=\seq{}$ is also minimal. \end{remark}

\begin{lemma}\label{res-learnrec}Suppose that $\xi:X^\NN\to X^\NN$ is defined by
\begin{equation*}\xi(u):=\initSeg{u}{\xi_0(u)}\ast\xi_1(u)\end{equation*}
where $\xi_0:X^\NN\to\NN$ and $\xi_1:X^\NN\to X^\NN$, and that $\xi$ satisfies
\begin{equation}\label{eqn-learnrec-i}(\forall u)(\neg C_0(u)\to u_{\xi_0(u)}\rhd \xi_1(u)_0\wedge P_0(\xi(u))).\end{equation}
Moreover, suppose that $\phi:X^\NN\to\NN$ is an additional functional which satisfies
\begin{equation}\label{eqn-learnrec-ii}(\forall u)(P_0(u)\to P_0(\specs{u}{\phi})).\end{equation}
For any $u:X^\NN$, the controlled learning procedure $\lrnc{\phi}{\xi}{C_0}[u]$ starting at $u$ denotes the sequence $(u_i)_{i\in\NN}$ given by
\begin{equation*}u_0:=u \mbox{ \ \ \ and \ \ \ }u_{i+1}:=\begin{cases}\specs{u_i}{\phi} & \mbox{if $C_0(\specs{u_i}{\phi})$}\\ \xi(\specs{u_i}{\phi}) & \mbox{otherwise}.\end{cases}\end{equation*}
Then provably from $\CONT$, firstly there always exists some $k$ such that $C_0(\specs{u_k}{\phi})$ holds, and we call $\specs{u_k}{\phi}$ for the minimal such $u_k$ the limit of $\lrnc{\phi}{\xi}{C_0}[u]$, which we denote by
\begin{equation*}\limc{\phi}{\xi}{C_0}{u},\end{equation*} 
secondly the functional $\lambda u.\limc{\phi}{\xi}{C_0}{u}$ is definable in $\systemT+(\eorec_{(X,\rhd)})$, and finally we have
\begin{equation}\label{eqn-learnrec-iii}(\forall u)(P_0(u)\to P_0(\limc{\phi}{\xi}{C_0}{u})\wedge C_0(\limc{\phi}{\xi}{C_0}{u})).\end{equation}
\end{lemma}

\begin{proof}We first formally construct the limit functional, which the reader can skip if they like since this is nothing more that a somewhat intricate unwinding of definitions. First, we define $L^\phi_{\xi,C_0}:X^\NN\to (X^\NN)^\ast$ by $L^\phi_{\xi,C_0}(u):=\eorec^{H,F}_{(X,\rhd),(X^\NN)^\ast}(u)$ where
\begin{equation*}\begin{aligned}F\alpha&:= \phi(\alpha_0)\\ 
H\alpha&:=\begin{cases}\seq{\alpha_0} & \mbox{if $C_0(\alpha_0)$}\\ \seq{\alpha_0}\ast \alpha_1\xi_0(\alpha_0)\xi_1(\alpha_0) & \mbox{otherwise}.\end{cases}\end{aligned}\end{equation*}
Here we denote by $\alpha_0$ the sequence $\lambda n.\pi_0\alpha(n)$ and similarly for $\alpha_1$. Then unwinding the definition, we have $L^\phi_{\xi,C_0}(u)=H(\specs{\alpha}{F})$ for $\alpha=\lambda n.\pair{u_n,\lambda v.L^\phi_{\xi,C_0}(\initSeg{u}{n}\ast v)\mbox{ if $v_0\lhd u_n$}}$. But since $F\alpha$ depends only on the first component $\alpha_0=u$, we have (by Lemma \ref{res-spec-prop}) $\specs{\alpha}{F}=\exts{\alpha}{m_0}$ where $m_0$ is the least number satisfying $F(\exts{\alpha}{m_0})=\phi(\exts{u}{m_0})<m_0$. This means that $\specs{u}{\phi}=\exts{u}{m_0}$ and so $(\specs{\alpha}{F})_0=(\exts{\alpha}{m_0})_0=\exts{u}{m_0}=\specs{u}{\phi}$ and $(\specs{\alpha}{F})_1=\lambda n<m_0,v.L^\phi_{\xi,C_0}(\initSeg{u}{n}\ast v)\mbox{ if $v_0\lhd u_n$}$, and so
\begin{equation*}H(\specs{\alpha}{F})=\begin{cases}\seq{\specs{u}{\phi}} & \mbox{if $C_0(\specs{u}{\phi})$}\\ \seq{\specs{u}{\phi}}\ast (\specs{\alpha}{F})_1\xi_0(\specs{u}{\phi})\xi_1(\specs{u}{\phi}) & \mbox{otherwise}.\end{cases}\end{equation*}
But now by (\ref{eqn-learnrec-i}), if $\neg C_0(\specs{u}{\phi})$ then $(\specs{u}{\phi})_{\xi_0(\specs{u}{\phi})}\rhd \xi_1(\specs{u}{\phi})_0$, and since $0_X$ was chosen to be minimal with respect to $\rhd$ (cf. Remark \ref{rem-zero}) this can only mean that $\xi_0(\specs{u}{\phi})<m_0$ (else $(\specs{u}{\phi})_{\xi_0(\specs{u}{\phi})}=0_X$) and therefore $(\specs{u}{\phi})_{\xi_0(\specs{u}{\phi})}=u_{\xi_0(\specs{u}{\phi})}$. Substituting all this information into the $(\specs{\alpha}{F})_1$ we have
\begin{equation*}\begin{aligned}(\specs{\alpha}{F})_1\xi_0(\specs{u}{\phi})\xi_1(\specs{u}{\phi})=L^\phi_{\xi,C_0}(\initSeg{\specs{u}{\phi}}{\xi_0(\specs{u}{\phi})}\ast \xi_1(\specs{u}{\phi}))=L^\phi_{\xi,C_0}(\xi(\specs{u}{\phi})).\end{aligned}\end{equation*}
So to summarise, the functional $L_{\xi,C_0}^\phi$ satisfies (repressing subscripts)
\begin{equation*}L_{\xi,C_0}^\phi(u)=\begin{cases}\seq{\specs{u}{\phi}} & \mbox{if $C_0(\specs{u}{\phi})$}\\ \seq{\specs{u}{\phi}}\ast L_{\xi,C_0}^\phi(\xi(\specs{u}{\phi})) & \mbox{otherwise},\end{cases}\end{equation*}
and so by induction on the length of $L_{\xi,C_0}^\phi(u)$ one establishes that $\limc{\phi}{\xi}{C_0}{u}$ exists and is the last element of $L_{\xi,C_0}^\phi(u)$. 

To verify (\ref{eqn-learnrec-iii}) is similar to the proof of Lemma \ref{res-learn}. We first show by induction that if $P_0(u)$ holds then $P_0(\specs{u_i}{\phi})$ holds for all $i\in\NN$. For $i=0$ this follows by (\ref{eqn-learnrec-ii}) applied to $u=u_0$, and otherwise if $P_0(\specs{u_i}{\phi})$ is true then either $\specs{u_{i+1}}{\phi}=\specs{\specs{u_i}{\phi}}{\phi}=\specs{u_i}{\phi}$ by Lemma \ref{res-spec-prop} or $\neg C(\specs{u_i}{\phi})$ and then by (\ref{eqn-learnrec-i}) we have $P_0(u_{i+1})$ and hence $P_0(\specs{u_{i+1}}{\phi})$ by (\ref{eqn-learnrec-ii}). Therefore, if $\specs{u_k}{\phi}$ is the limit of the learning procedure, then $P_0(\specs{u_k}{\phi})\wedge C_0(\specs{u_k}{\phi})$ holds, and we're done. \end{proof}

Our final step is now to produce a realizer for the functional interpretation of $\ZL$. Let's briefly recall from Section \ref{sec-ZLa} what this means: We are given as input a sequence $\bar u$ (we use this new notation as we want $u$ to denote a separate variable below), a pair of functionals $M,N:X^\NN\to (\NN\to X^\NN\to \NN)\to \NN$, together with $W:X^\NN\to (\NN\to X^\NN\to \NN)\to X^\NN$, and we must produce some $n:\NN$, $v: X^\NN$ and $\gamma:(X^\NN\to\NN)^\NN$ satisfying
\begin{equation}\label{eqn-zorn-functional}\bar P(\bar u,n)\to \bar P(v,Nv\gamma)\wedge C(v,\gamma,Mv\gamma,Wv\gamma),\end{equation}
where as before $C(v,\gamma,m,w):\equiv w_0\lhd v_m\to \neg \bar P(\initSeg{v}{m}\ast w,\gamma mw)$. We first need some definitions. Define the functional $\Psi^N:X^\NN\to\NN$ by
\begin{equation*}\Psi^N(u):=\eorec^{\tilde N,\tilde N}_{(X,\rhd),\NN}(u)\end{equation*}
where $\tilde N\alpha:= N\alpha_0\alpha_1$. Using $\Psi$, for each $u:X^\NN$ define $\gamma_u:\NN\to X^\NN\to \NN$ by
\begin{equation*}\gamma_u:=\lambda n,v\; . \; \Psi^N(\initSeg{u}{n}\ast v)\mbox{ if $v_0\lhd u_n$}.\end{equation*}
Finally, define parameters $\phi:X^\NN\to\NN$, $\xi_0:X^\NN\to\NN$ and $\xi_1: X^\NN\to X^\NN$, together with predicates $P_0$ and $C_0$, by
\begin{equation*}\begin{aligned}\phi(u)&:=Nu\gamma_u\\
\xi_0(u)&:= Mu\gamma_u\\
\xi_1(u)&:= Wu\gamma_u\\
P_0(u)&:\equiv \bar P(\bar u,\Psi^N(\bar u))\to\bar P(u,\Psi^N (u))\\
C_0(u)&:\equiv C(u,\gamma_u,Mu\gamma_u,Wu\gamma_u).\end{aligned}\end{equation*}
Now it is perhaps becoming clear to the reader what will come next: We will set up a controlled learning procedure $\lrnc{\phi}{\xi}{C_0}[\bar u]$ on these parameters exactly as in Lemma \ref{res-learnrec}, and the limit $v:=\specs{u_k}{\phi}$ of $\lrnc{\phi}{\xi}{C_0}[\bar u]$ will satisfy $P_0(v)\wedge C_0(v)$, or in other words, $\bar P(\bar u,\Psi^N(\bar u))\to\bar P(v,\Psi^N (v))$ and $C(v,\gamma_v,Mv\gamma_v,Wv\gamma_v)$, from which we will be able to construct our realizer of $\ZL$. Let's make this formal.

\begin{theorem}\label{res-ZL-main}Let $n,v$ and $\gamma$ be defined in terms of $\bar{u},N,M$ and $W$ by
\begin{equation*}\begin{aligned}n&:=\Psi^N(\bar u)\\
v&:=\limc{\phi}{\xi}{C_0}{\bar u}\\
\gamma &:=\gamma_v.\end{aligned}\end{equation*}
Then (provably in $\CONT$) these satisfy (\ref{eqn-zorn-functional}) and therefore solve the functional interpretation of $\ZL$.  \end{theorem}

\begin{proof}We use Lemma \ref{res-learnrec}, which means that we must check that each of (\ref{eqn-learnrec-i}) and (\ref{eqn-learnrec-ii}) hold for our choice of $P_0$ and $C_0$, i.e. we must prove
\begin{equation}\label{eqn-ZL-i}(\forall u)[\neg C(u,\gamma_u,Mu\gamma_u,Wu\gamma_u)\to u_{\xi_0(u)}\rhd \xi_1(u)_0\wedge (\bar P(\bar u,\Psi^N(\bar u))\to\bar P(\xi(u),\Psi^N (\xi(u)))) ]\end{equation}
and
\begin{equation}\label{eqn-ZL-ii}(\forall u)[(\bar P(\bar u,\Psi^N(\bar u))\to\bar P(u,\Psi^N (u)))\to (\bar P(\bar u,\Psi^N(\bar u))\to\bar P(\specs{u}{\phi},\Psi^N (\specs{u}{\phi})))].\end{equation} 
For the first condition, note that $\neg C(u,\gamma_u,Mu\gamma_u,Wu\gamma_u)$ is just $\neg C(u,\gamma_u,\xi_0(u),\xi_1(u))$, which implies both $\xi_1(u)_0\lhd u_{\xi_0(u)}$ and $\bar P(\initSeg{u}{\xi_0(u)}\ast\xi_1(u),\gamma_u\xi_0(u)\xi_1(u))$. But since
\begin{equation*}\initSeg{u}{\xi_0(u)}\ast\xi_1(u)=\xi(u)\mbox{ \ \ \ and \ \ \ }\gamma_u\xi_0(u)\xi_1(u)=\Psi^N(\xi(u))\end{equation*}
we have established $\bar P(\xi(u),\Psi^N(\xi(u))$, and hence the conclusion of (\ref{eqn-ZL-i}).

The second condition is more subtle: Either $\neg \bar P(\bar u,\Psi^N(\bar u))$ and we're done, or it suffices to prove $\bar P(u,\Psi^N(u))\to \bar P(\specs{u}{\phi},\Psi^N(\specs{u}{\phi}))$. We now need to unwind the definition of $\Psi^N(u)$: First note that we have 
\begin{equation*}\Psi^N(u)=\tilde N(\specs{\alpha}{\tilde N})\end{equation*}
where (using the definition of $\gamma_u$)
\begin{equation*}\alpha:=\lambda n.\pair{u_n,\lambda v\; . \; \Psi^N(\initSeg{u}{n}\ast v)\mbox{ if $v_0\lhd u_n$}}=\lambda n.\pair{u_n,\gamma_{u,n}}.\end{equation*}
and so in particular we have (by the definitions of $\tilde N$ and $\phi$)
\begin{equation}\label{eqn-Nphi}\tilde N(\exts{\alpha}{m})=N(\exts{u}{m})(\exts{\gamma_u}{m})=N(\exts{u}{m})(\gamma_{\exts{u}{m}})=\phi(\exts{u}{m})\end{equation}
where for the central equality we use the assumption that $0_X$ is minimal with respect to $\rhd$ and so
\begin{equation*}\begin{aligned}\gamma_{\exts{u}{m}}=\lambda n, v.\Psi^N(\initSeg{\exts{u}{m}}{n}\ast v)\mbox{ if $v_0\lhd (\exts{u}{m})_n$}&=\lambda n\; . \; \begin{cases}\lambda v.\Psi^N(\initSeg{u}{n}\ast v)\mbox{ if $v_0\lhd u_n$} & \mbox{if $n<m$} \\ \lambda v.\Psi^N(\initSeg{u}{n}\ast v)\mbox{ if $v_0\lhd 0_X$} & \mbox{otherwise}\end{cases}\\
&=\lambda n\; . \; \begin{cases}\lambda v.\gamma_unv & \mbox{if $n<m$} \\ \lambda v.0 & \mbox{otherwise}\end{cases}\\
&=\exts{\gamma_u}{m}.\end{aligned}\end{equation*}
Let $m_0$ be the least number such that $\tilde N(\exts{\alpha}{m_0})<m_0$ (which exists since we are assuming $\CONT$), and therefore by (\ref{eqn-Nphi}) the also the least number such that $\phi(\exts{u}{m_0})<m_0$. Then by Lemma \ref{res-spec-prop} we have that for all $u$:
\begin{equation}\label{eqn-Psi}\Psi^N(u)=\tilde N(\specs{\alpha}{\tilde N})\stackrel{L.\; \ref{res-spec-prop} (\ref{item-spec-propi})}{=}\tilde N(\exts{\alpha}{m_0})\stackrel{(\ref{eqn-Nphi})}{=}\phi(\exts{u}{m_0})\stackrel{L.\; \ref{res-spec-prop} (\ref{item-spec-propi})}{=}\phi(\specs{u}{\phi})=N\specs{u}{\phi}\gamma_{\specs{u}{\phi}}.\end{equation}
In particular, by Lemma \ref{res-spec-prop} (\ref{item-spec-propiii}) we have
\begin{equation}\label{eqn-idemp}\Psi^N(\specs{u}{\phi})\stackrel{(\ref{eqn-Psi})}{=}\phi(\specs{\specs{u}{\phi}}{\phi})\stackrel{L.\; \ref{res-spec-prop} (\ref{item-spec-propiii})}{=}\phi(\specs{u}{\phi})\stackrel{(\ref{eqn-Psi})}{=}\Psi^N(u)\end{equation}
and
\begin{equation}\label{eqn-less}\Psi^N(u)\stackrel{(\ref{eqn-Psi})}{=}\tilde N(\exts{\alpha}{m_0})<m_0\end{equation}
and therefore
\begin{equation*}\initSeg{\specs{u}{\phi}}{\Psi^N(\specs{u}{\phi})}\stackrel{(\ref{eqn-idemp})}{=}\initSeg{\specs{u}{\phi}}{\Psi^N(u)}\stackrel{L.\; \ref{res-spec-prop}(\ref{item-spec-propi})}{=}\initSeg{\exts{u}{m_0}}{\Psi^N(u)}\stackrel{(\ref{eqn-less})}{=}\initSeg{u}{\Psi^N(u)}\end{equation*}
and thus $P(\initSeg{u}{\Psi^N(u)})$ implies $P(\initSeg{\specs{u}{\phi}}{\Psi^N(\specs{u}{\phi})})$, which establishes (\ref{eqn-ZL-ii}).

It now follows from Lemma \ref{res-learnrec} that
\begin{equation*}P_0(\bar u)\to P_0(v)\wedge C_0(v)\end{equation*}
and since $P_0(\bar u)$ is trivially true we have established
\begin{equation*}\bar P(\bar u,\Psi^N(\bar u))\to\bar P(v,\Psi^N (v))\mbox{ \ \ \ and \ \ \ }C(v,\gamma_v,Mv\gamma_v,Wv\gamma_v).\end{equation*}
We can now prove (\ref{eqn-zorn-functional}). Suppose that $\bar P(u,n)$ holds. Then since $n=\Psi^N(\bar u)$ from the left hand side we have $\bar P(v,\Psi^N(v))$. Now, since $v=\specs{u_k}{\phi}$ for some element in the learning procedure $\lrnc{\phi}{\xi}{C_0}[\bar u]$, by (\ref{eqn-idemp}) we have
\begin{equation*}\Psi^N(v)=\Psi^N(\specs{u_k}{\phi})=\phi(\specs{u_k}{\phi})=N\specs{u_k}{\phi}\gamma_{\specs{u_k}{\phi}}=Nv\gamma_v=Nv\gamma\end{equation*}
and so we have established $\bar P(v,Nv\gamma)$. Then, since $C(v,\gamma,Mv\gamma,Wv\gamma)$ is given to us automatically, we have proven
\begin{equation*}\bar P(u,n)\to \bar P(v,Nv\gamma)\wedge C(v,\gamma,Mv\gamma,Wv\gamma)\end{equation*}
which is exactly (\ref{eqn-zorn-functional}), and so we're done.\end{proof}

The results of this section mark the technical climax of the paper, and in particular form our broadest and most widely applicable contribution. While the proofs above are perhaps somewhat difficult to navigate, it is important to emphasise that most of the technical details are bureaucratic in nature, in the unwinding of all the definitions and the careful use of Lemma \ref{res-spec-prop}. The \emph{intuition} behind our realizer, on the other hand, should hopefully be clear from the somewhat more informal discussion in Section \ref{sec-ZLa}. In any case, now that the hard work is done, a computational interpretation of Nash-William's proof of Higman's lemma follows relatively easily.

\section{Interpreting the proof of $\WQO(=_{\BB,\ast})$}
\label{sec-higman}

We are now finally ready to produce our realizer for the statement that $=_{\BB^\ast}$ is a WQO. In fact we do something more general, namely give a computational interpretation to the proof that $\sWQO(\preceq)\to\WQO(\preceq_\ast)$, which is valid for \emph{any} well quasi-order $\preceq$. Recall that the functional interpretation of $\sWQO(\preceq)$ is given by
\begin{equation}\label{eqn-wqoseq-nd}(\forall x^{X^\NN},\omega^{(\NN\to\NN)\to\NN})(\exists g)(\forall i<j\leq \omega g)(g(i)<g(j)\wedge x_{g(i)}=x_{g(j)}).\end{equation}
Lemma \ref{res-higman-fiddly} makes precise exactly how we use the assumption $\sWQO(\preceq)$ to prove $\WQO(\preceq_\ast)$: Namely given a hypothetical minimal bad sequence of words $v$, we take the sequence $\bar{v}$ and require that our monotone subsequence $g$ be valid up to the point $k$, where $k$ is such that $P(w,k)$ holds for $w=\initSeg{v}{g(0)}\ast\tilde v_{g}$ as defined in Lemma \ref{res-higman-fiddly}. If minimality of $v$ is witnessed by some functional $\gamma$ then such a $k$ would be given by $\gamma(g0)(\tilde v_g)$. This motivates the following:
\begin{lemma}\label{res-higman-lemma}Suppose that $G$ is a realizer for (\ref{eqn-wqoseq-nd}):
\begin{equation}\label{eqn-resseq-nd}(\forall x,\omega)(\forall i<j<\omega G_{x,\omega})(G_{x,\omega}(i)<G_{x,\omega}(j)\wedge x_{G_{x,\omega}(i)}\preceq x_{G_{x,\omega}(j)}).\end{equation}
Then from this we can construct a functional $H:X^\NN\to (\NN\to X^\NN\to \NN)\to \NN^\NN$ satisfying
\begin{equation}\label{eqn-wqoseqs-nd}(\forall v,\gamma)(\forall i<j<\gamma(H_{v,\gamma}(0))(\tilde v_{H_{v,\gamma}}))(H_{v,\gamma}(i)<H_{v,\gamma}(j)\wedge \bar v_{H_{v,\gamma}(i)}\preceq \bar v_{H_{v,\gamma}(j)})\end{equation}
where $v_{H_{v,\gamma}}$ is shorthand for $\lambda i.v_{v_{H_{v,\gamma}}(i)}$, and $\tilde v,\bar v$ are defined as in Lemma \ref{res-higman-fiddly}.\end{lemma}
\begin{proof}Define $\omega_{v,\gamma}g:= \gamma(g(0))(\tilde v_g)$ and then $H_{v,\gamma}:=G_{\bar v,\omega_{v,\gamma}}$. Then (\ref{eqn-wqoseqs-nd}) follows directly from (\ref{eqn-resseq-nd}).\end{proof}

We will now give a computational version of Lemma \ref{res-higman-fiddly} as a whole:

\begin{lemma}\label{res-higman-fiddlylem}Suppose that $H$ satisfies (\ref{eqn-wqoseqs-nd}) and that $v$ and $\gamma$ satisfy $C(v,\gamma,H_{v,\gamma}(0),\tilde v_{H_{v,\gamma}})$, which analogously to before abbreviates
\begin{equation}\label{eqn-minim}\tilde v_{H_{v,\gamma}(0)}\lhd v_{H_{v,\gamma}(0)}\to \underbrace{(\exists i<j<\gamma (H_{v,\gamma}(0))(\tilde v_{H_{v,\gamma}}))((\initSeg{v}{H_{v,\gamma}(0)}\ast \tilde v_{H_{v,\gamma}})_i\preceq_\ast (\initSeg{v}{H_{v,\gamma}(0)}\ast \tilde v_{H_{v,\gamma}})_j)}_{\neg \bar P(\initSeg{v}{H_{v,\gamma}(0)}\ast \tilde v_{H_{v,\gamma}},\gamma (H_{v,\gamma}(0))(\tilde v_{H_{v,\gamma}}))}.\end{equation}
Then we have $\neg \bar P(v,H_{v,\gamma}(\gamma(H_{v,\gamma}(0))(\tilde v_{H_{v,\gamma}}))+2)$ i.e.
\begin{equation*}(\exists i<j<H_{v,\gamma}(\gamma(H_{v,\gamma}(0))(\tilde v_{H_{v,\gamma}}))+2)(v_i\preceq_\ast v_j).\end{equation*}\end{lemma}

\begin{proof}This follows directly from Lemma \ref{res-higman-fiddly}. First of all, we define $g:=H_{v,\gamma}$, then the sequence $w$ in Lemma \ref{res-higman-fiddly} becomes identified with $\initSeg{v}{H_{v,\gamma}(0)}\ast\tilde v_{H_{v,\gamma}}$, and so setting $k:=\gamma(H_{v,\gamma}(0))(\tilde v_{H_{v,\gamma}})$, the equation (\ref{eqn-minim}) is just (\ref{eqn-higman-min}), while (\ref{eqn-wqoseqs-nd}) is just (\ref{eqn-higman-mon}), and so by the lemma there exists some $i<j<g(k)+2$ such that $v_i\preceq_\ast v_j$. But since $g(k)+2=H_{v,\gamma}(\gamma(H_{v,\gamma}(0))(\tilde v_{H_{v,\gamma}}))+2$ we're done.\end{proof}

What we have shown above is that if $C(v,\gamma,H_{v,\gamma}(0),\tilde v_{H_{v,\gamma}})$ then $\neg \bar P(v,H_{v,\gamma}(\gamma(H_{v,\gamma}(0))(\tilde v_{H_{v,\gamma}}))+2)$, or in other words,
\begin{equation*}\bar P(v,H_{v,\gamma}(\gamma(H_{v,\gamma}(0))(\tilde v_{H_{v,\gamma}}))+2)\wedge C(v,\gamma,H_{v,\gamma}(0),\tilde v_{H_{v,\gamma}})\end{equation*}
must be false. But this is just the conclusion of the function interpretation of $\ZL$ for $Nv\gamma=H_{v,\gamma}(\gamma(H_{v,\gamma}(0))(\tilde v_{H_{v,\gamma}}))+2$, $Mv\gamma=H_{v,\gamma}(0)$ and $Wv\gamma=\tilde v_{H_{v,\gamma}}$, and so for any $v,\gamma$ and $n$ satisfying (\ref{eqn-zorn-functional}) we must have $\neg \bar P(u,n)$, which is exactly what we want! Let's make this formal.
\begin{theorem}\label{res-higman-comp}Define $N,M:(X^\ast)^\NN\to (\NN\to (X^\ast)^\NN\to\NN)\to\NN$ and $W:(X^\ast)^\NN\to (\NN\to(X^\ast)^\NN\to\NN)\to (X^\ast)^\NN$ by
\begin{equation*}\begin{aligned}Nv\gamma&:=H_{v,\gamma}(\gamma(H_{v,\gamma}(0))(\tilde v_{H_{v,\gamma}}))+2\\
Mv\gamma &:=H_{v,\gamma}(0)\\
Wv\gamma&:=\tilde v_{H_{v,\gamma}}\end{aligned}\end{equation*}
where $H$ is some functional which satisfies (\ref{eqn-wqoseqs-nd}). Define $P(s):\equiv (\forall i<j<|s|)(s_i\npreceq_\ast s_j)$ so that $\bar P(u,k)\equiv (\forall i<j<k)(u_i\npreceq_\ast u_j)$, and let $n,v$ and $\gamma$ be such that they satisfy the functional interpretation of $\ZL$ relative to $N,M,W$ defined above:
\begin{equation}\label{eqn-ZLnd-used}\bar P(u,n)\to \bar P(v,Nv\gamma)\wedge C(v,\gamma,Mv\gamma,Wv\gamma)\end{equation}
Then we have $\neg \bar P(u,n)$ and hence
\begin{equation*}(\forall u)(\exists i<j<n)(u_i\preceq_\ast u_j).\end{equation*}\end{theorem} 

\begin{proof}As shown above, if $H$ satisfies $(\ref{eqn-wqoseqs-nd})$ and $C(v,\gamma,Mv\gamma,Wv\gamma)$, then by Lemma \ref{res-higman-fiddlylem} we have $\neg \bar P(v,Nv\gamma)$, which implies $\neg (\bar P(v,Nv\gamma)\wedge C(v,\gamma,Mv\gamma,Wv\gamma))$, and so by the contrapositive of (\ref{eqn-ZLnd-used}) we have $\neg\bar P(u,n)$.\end{proof}

Therefore any program which computes $v,\gamma$ and $n$ on any $u,N,M$ and $W$, in particular that of Theorem \ref{res-ZL-main}, can be converted to a program which realizes $\WQO(\preceq_\ast)$:

\begin{corollary}\label{res-higman-prog}Suppose that $H$ satisfies (\ref{eqn-wqoseqs-nd}), and define $Nv\gamma:=H_{v,\gamma}(\gamma(H_{v,\gamma}(0))(\tilde v_{H_{v,\gamma}}))+2$. Then provably in $\CONT$ the functional $\Phi:(X^\ast)^\NN\to\NN$ defined by
\begin{equation*}\Phi(u):=\Psi^N(u)\end{equation*}
where $\Psi^N(u)$ is defined as in Theorem \ref{res-ZL-main} witnesses $\WQO(\preceq_\ast)$ i.e.
\begin{equation*}(\exists i<j<\Phi(u))(u_i\preceq_\ast u_j).\end{equation*}\end{corollary}

\begin{corollary}\label{res-prog}Let $H$ be defined as in Lemma \ref{res-higman-lemma} for $G$ as defined in Section \ref{sec-sWQO}. Then $\Phi:(\BB^\ast)^\NN\to\NN$ as defined in Corollary \ref{res-higman-prog} witnesses $\WQO(=_{\BB,\ast})$ i.e.
\begin{equation*}(\exists i<j<\Phi(u))(u_i=_{\BB,\ast} u_j).\end{equation*}\end{corollary} 

\begin{remark}To construct our realizer for $\WQO(\preceq_\ast)$ we have only used the first component $n$ of the full functional interpretation $n,v,\gamma$ of $\ZL$. This makes sense: We actually prove via Lemma \ref{res-higman-fiddlylem} that $(\exists v,\gamma)Q(v,\gamma)\to\bot$ where $Q(v,\gamma)$ abbreviates the conclusion of (\ref{eqn-ZLnd-used}), and so to realize Higman's lemma we in fact only need to produce some functional $\Phi$ such that $\bar P(u,\Phi(u))\to (\exists v,\gamma)Q(v,\gamma)$, and so the full computational interpretation of $\ZL$ via learning procedures was not strictly necessary. However, this simply emphasises the fact that we have achieved much more that a realizer for Higman's lemma - Theorem \ref{res-ZL-main} allows us to extract a program from \emph{any} proof which uses $\ZL$, and in general this program may well need a specific $v$ and $\gamma$ satisfying $Q(v,\gamma)$, even though here that was not the case. There is a further point to be made in this direction - namely that the concrete witnesses for $v$ and $\gamma$ enables us to \emph{verify} our realizer $\Phi(u)$ in a quantifier-free theory, a fact that is relevant to those inclined towards foundational issues.  \end{remark}

Before we conclude, it is worth pausing for a moment and trying to explain from an algorithmic point of view what the realizer we get in Corollary \ref{res-higman-prog} actually does. Note that all of the following is essentially just an informal recapitulation of ideas contained in the preceding results. Roughly speaking, $\Phi(u)$ encodes a program which works by recursion on the lexicographic ordering $\rrhd$. First, it finds the point $m_0$ such that 
\begin{equation*}N(\exts{u}{m_0})(\exts{\gamma_u}{m_0})<m_0\end{equation*}
where $\gamma_u:=\lambda n,w\; . \; \Phi(\initSeg{u}{n}\ast w)\mbox{ if $w_0\lhd u_n$}$, and so in particularly it only looks at the sequence $u$ at points $n<m_0$. For simplicity let's define $u',\gamma':=\exts{u}{m_0},\exts{\gamma_u}{m_0}$. Now, using any program $H$ which realizes $\sWQO(\preceq)$ on the sequence $\bar{u}'$ we find a sufficiently large approximation $H_{u',\gamma'}$ to a constant subsequence, which works up to the point $\gamma'(H_{u',\gamma'}(0))( \tilde u'_{H_{u',\gamma'}})$.

Now if $\tilde u'_{H_{u',\gamma'}(0)}\lhd u'_{H_{u',\gamma'}(0)}$ then we must have $H_{u',\gamma'}(0)<m_0$ (using our assumption that $0_X$ is chosen to be minimal with respect to $\lhd$) and so  $\gamma'(H_{u',\gamma'}(0))( \tilde u'_{H_{u',\gamma'}})=\Phi(\initSeg{u'}{H_{u',\gamma'}(0)}\ast \tilde u'_{H_{u',\gamma'}})$. Assuming inductively that this returns a bound for $\initSeg{u'}{H_{u',\gamma'}(0)}\ast \tilde u'_{H_{u',\gamma'}}$ being a good sequence, then using reasoning as in Lemma \ref{res-higman-fiddly} this means that $u'$ becomes a good sequence before point $H_{u',\gamma'}(\Phi(\initSeg{u'}{H_{u',\gamma'}(0)}\ast \tilde u'_{H_{u',\gamma'}}))+2=Nu'\gamma'$. But since $\Phi(u)=Nu'\gamma'<m_0$ and $u'=\exts{u}{m_0}$ then this means also that $u$ is good before $\Phi(u)$.

To verify that $\Phi(\initSeg{u'}{H_{u',\gamma'}(0)}\ast \tilde u'_{H_{u',\gamma'}})$ returns a bound, we can repeat this argument for $u_1:=\initSeg{u'}{H_{u',\gamma'}(0)}\ast \tilde u'_{H_{u',\gamma'}}$, and we end up with a learning procedure as in Section \ref{sec-ZLb}. Eventually, this learning procedure will terminate with a minimal sequence $v$ such that $\Phi(v)$ is guaranteed to witnesses that $v$ is good.

\section{Conclusion}
\label{sec-conclusion}

I will conclude by tying up everything that we've done and outlining some directions for future work. On the route to Corollary \ref{res-prog} we took what we hope was a pleasant and instructive detour through many different areas which connect proof theory and well quasi-order theory, the most important of which I will now summarise.

Right at the start, in Chapters \ref{sec-wqo}-\ref{sec-formal}, we discussed various nuances that arise when giving an axiomatic formalisation of results in WQO theory, in particular how the distinction between dependent choice or Zorn's lemma plays an important role in the context of program extraction. The full formalisation of Nash-Williams' minimal bad sequence argument has already been studied in e.g. \cite{Seisenberger(2003.0)} (in \textsc{Minlog}) and \cite{Sternagel(2013.0)} (in Isabelle/HOL), and we hope that our formal proof sketched in Chapter \ref{sec-formal} may prove informative to those working in a more hands-on manner on the formalisation of WQO theory in proof assistants.

In Chapters \ref{sec-goedel}-\ref{sec-sWQO} we took the opportunity to present G\"{o}del's functional interpretation in a way that would appeal to readers not already familiar with it. In Section \ref{sec-goedel-meaning} we placed particular emphasis on explaining how the interpretation behaves in \emph{practice}, and in this vein we gave a carefully worked out case study in Section \ref{sec-sWQO}, which also formed a key Lemma in our proof of $\WQO(=_{\BB,\ast})$. It is my sincere hope that these chapters will be a general help to those interested in how proof interpretations work, independently of the rest of the paper.

Chapters \ref{sec-ZLa} and \ref{sec-ZLb} contain our main technical contribution, namely the solution of the functional interpretation of $\ZL$. While as a direct consequence this enables us to extract a program witnessing $\WQO(=_{\BB,\ast})$, our work in these chapters is much broader, and provides us with a method of giving a computational interpretation to \emph{any} proof that can be formalised in $\PAomega+\QFAC+\ZL$, where moreover $\ZL$ can involve any relation $(X,\rhd)$ which is provably wellfounded. In particular, this paves the way for the extraction of programs from much more complex proofs in WQO theory, such as Kruskal's theorem, and we intend to address this in future work.

Hidden in Chapters \ref{sec-ZLa} and \ref{sec-ZLb} is also a small extension of my work on \emph{learning procedures} \cite{Powell(2016.0)}. While in this paper they play the role of making our computational interpretation of $\ZL$ more intuitive, Lemma \ref{res-learnrec} is of interest in its own right, as it demonstrates that we can extend the notion of a learning procedure as introduced in \cite{Powell(2016.0)} to the non-wellfounded ordering $\rrhd$. We anticipate that a number of variants of open induction or Zorn's lemma over $\rrhd$ could be given computational interpretations by appealing directly to Lemma \ref{res-learnrec} as a intermediate result, and this was part of our motivation for stating it explicitly here.

Sandwiched between these sections is Chapter \ref{sec-OR}, which itself forms a small essay on higher-type computability theory, and the various ways of carrying out recursion over $\rrhd$ in the continuous functionals. There are a number of interesting questions to be answered in this direction. Firstly, what is the relationship between $\eorec$, Berger's open recursion and the many variants of bar recursion which have been devised in the context of proof theory? I have already shown that Berger's open recursion is primitive recursively equivalent to \emph{modified bar recursion} \cite{BergOli(2005.0)} and thus strictly stronger than Spector's original bar recursion, but I conjecture that in contrast, $\eorec$ is equivalent to Spector's bar recursion and thus weaker than Berger's open recursion. This would also imply that $\eorec$ exists in the type structure of $\modmaj$ strongly majorizable functionals, and so does not necessarily rely on continuity to be a wellfounded form of recursion. It would be interesting to explore some of these issues in the future.

Finally, we should not forget that in Chapter \ref{sec-higman} we gave a new program which witnesses Higman's lemma, that works not just for the $=_\BB$, but for \emph{any} WQO for which a realizer of $\sWQO(\preceq)$ can be given. Moreover, in contrast to \cite{Powell(2012.0)}, this realizer encodes a recursive algorithm which seems to do something fundamentally intuitive. The precise relationship between this algorithm and the many others which have been offered over the years is a question we leave open for now, although the presence of the control functional in the explicit open recursor leads me to conjecture that it is genuinely different to most of them. But for now we simply hope that, among other things, we have provided a little more insight into the computational meaning of Nash-Williams' elegant classical proof.\\

\noindent\textbf{Acknowledgements.} In developing the ideas of this paper I have benefited greatly from numerous illuminating conversations with Ulrich Berger, Paulo Oliva and Monika Seisenberger.

\bibliographystyle{plain}
\bibliography{/home/thomas/Dropbox/tp}
\end{document}